\documentclass[final,12pt]{clear2023} % Include author 
	\usepackage{amsfonts}
	\usepackage{dsfont}
	\usepackage{fixme}
	\fxsetup{status=draft}
	\usepackage{tikz} % nice language for creating drawings and diagrams
	
	\newtheorem{thm}{Theorem}

	\newtheorem{defn}[thm]{Definition}
	\newtheorem{exmp}[thm]{Example}

	\newcommand{\disjU}{\mathbin{\dot{\cup}}}

	\usepackage{tikz}
	\usepgflibrary{arrows}
	\usetikzlibrary{shapes}
	\usetikzlibrary{arrows.meta}
	\usetikzlibrary{positioning}
	
	\title[Instrumental Processes]{Instrumental Processes Using Integrated 
	Covariances}
	\usepackage{times}
	
	\clearauthor{%
		\Name{Søren Wengel Mogensen} 
		\Email{soren.wengel\_mogensen@control.lth.se}\\
		\addr Department of Automatic Control, Lund University, Sweden%
	}
	
\begin{document}
		\maketitle
		
		\begin{abstract}
		Instrumental variable methods are often used for parameter estimation 
		in the presence of confounding. They can also be applied in stochastic 
		processes. Instrumental variable analysis exploits moment 
		equations to obtain estimators for causal parameters. We show that in 
		stochastic processes one can find such moment equations using an 
		integrated covariance matrix. This provides new instrumental variable 
		methods, instrumental variable methods in a class of continuous-time 
		processes as well as a unified treatment of discrete- and 
		continuous-time processes.
		\end{abstract}
		

		\section{Introduction}\label{sec:intro}
		
		Instrumental variable (IV) techniques have a long history in economics, 
		engineering, and causal inference, even if each field has its own 
		standard formulation of the IV problem
		\citep{wright1928tariff, 
		reiersol1941confluence,reiersol1945confluence,sargan1958estimation,
		joseph1961plant, wong1966estimation, 
		wong1967identification}. Recent work \citep{thams2022} 
		formulates an instrumental variable problem in a (discrete-time) time 
		series model and 
		provides 
		a solution which employs conditional instruments 
		\citep{brito2002generalized}. \cite{thams2022} take 
		a \emph{variable-centric} approach in that they identify sets of 
		variables at different lags that satisfy conditions enabling 
		conditional instrumental variable techniques. This paper takes a 
		\emph{process-centric} approach, essentially by integrating out time. 
		The IV methods of this paper therefore only use integrated measures of 
		covariance of stochastic processes. The distinction between variable- 
		and process-centric will 
		be described in more detail in Section \ref{sec:ivm}. 
		
		The process-centric approach outlined in this paper is applicable to 
		discrete-time stochastic processes and can also be applied in 
		continuous time as we show using a class of point processes. The 
		estimand is slightly different than in existing methods, however, the 
		estimated parameter is easily interpretable and it summarizes the 
		strength 
		of the dependence between stochastic processes.
	
		As the paper uses both discrete- and continuous-time models, we only 
		use the term \emph{time series} to refer to stochastic processes in 
		discrete time. The paper is structured as follows. Section 
		\ref{sec:ivm} describes a 
		classical instrumental variable problem as well as the variable-centric 
		and process-centric approaches to IV estimation in time 
		series. Section \ref{sec:pm} describes the causal estimands 
		that 
		our IV equations identify. Section \ref{sec:ip} describes IV methods in 
		both linear Hawkes processes 
		and vector-autoregressive time series, however, we focus first on the 
		time series and the description of the linear Hawkes case is found in 
		Section \ref{sec:hawkes}. In both cases, we use an integrated 
		covariance matrix to obtain new IV results and there is a strong 
		conceptual similarity between the two, even though the interpretation 
		of 
		the parameters depends on the model class. We also generalize the 
		results in the time series setting to allow more general 
		confounding and more general instrumental processes (Section 
		\ref{sec:confoundingts}). 
		Section \ref{sec:est} 
		discusses estimation for time series models.

		\section{Instrumental Variable Methods}
		\label{sec:ivm}

								\begin{figure}
									%\begin{tabular}{cc}
									\begin{minipage}[.7\textheight]{0.4\linewidth}
										\centering
										\begin{tikzpicture}[scale=0.7]
										\tikzset{vertex/.style
											= 
											{shape=circle,draw,minimum 
											size=1.5em, 
												inner sep = 0pt}}
										\tikzset{edge/.style= {->,> = 
										latex',thick}}
										\tikzset{edgebi/.style= {<->,> = 
												latex', 
												thick}}
										\tikzset{every 
											loop/.style={min distance=8mm, 
											looseness=5}}
										\tikzset{vertexFac/.style= 	
											{shape=rectangle,draw,minimum 
											size=1.5em, 
												inner sep = 
												0pt}}
										
										% vertices
										%\draw [line 
										%width=35pt,opacity=0.1,
										% 
										%blue,line 
										%cap=round,rounded
										%corners] (0,0.5) 
										%-- (0,2) -- 
										%(-1.5,1.5) -- 
										%(0,0.5);
										\def\y{-2}						
										\node[vertex] (i) at (-4,0+\y) 	{$I$};
										\node[vertex] (a) at  	(-2,0+\y) {$A$};
										\node[vertex] (b) 	at  (2,0+\y) 	
										{$B$};
										\node[vertexFac] 	(u) 	at  	
										(0,1.25+\y)	{$U$};
										
										\node at 
										(-3.5,1.75+\y) {\Large 
											\textbf{A}};
										
										%edges
										
										\draw[edge] (i) to 
										(a);
										\draw[edge] (a) to 
										(b);				
										\draw[edge] (u) to 
										(a);
										\draw[edge] (u) to 
										(b);	
										
										% graph B
										\def\x{-6}
										\node[vertex] (i2) at (-4,0+\x) 	
										{$I$};
										\node[vertex] (a2) at  	(-2,0+\x) {$A$};
										\node[vertex] (b2) 	at  (2,0+\x) 	
										{$B$};
										\node[vertexFac] 	(u2) 	at  	
										(0,1.25+\x)	
										{$U$};
										
										\node at 
										(-3.5,1.75+\x) {\Large 
											\textbf{B}};
										
										%edges
										
										\draw[edge] (i2) to 
										(a2);
										\draw[edge] (a2) to 
										(b2);				
										\draw[edge] (u2) to 
										(a2);
										\draw[edge] (u2) to 
										(b2);
										\draw[edge, bend left] (b2) to 
										(a2);

										\end{tikzpicture}
									\end{minipage}\hspace{.05\linewidth}%
									\begin{minipage}{0.5\linewidth}
										\centering
										\begin{tikzpicture}[scale=0.9]
										\tikzset{vertex/.style = 
											{shape=circle,minimum 
												size=1.5em, 
												inner 
												sep = 0pt}}
										\tikzset{edge/.style = {->,> = latex', 
												thick}}
										\tikzset{edgebi/.style = {<->,> = 
												latex', 
												thick}}
										\tikzset{every loop/.style={min 
												distance=8mm, 
												looseness=5}}
										\tikzset{vertexFac/.style = 
											{shape=rectangle,draw,minimum 
												size=1.5em, 
												inner sep = 0pt}}
										
										% vertices
										%\draw [line width=35pt,opacity=0.1, 
										%blue,line 
										%cap=round,rounded
										%corners] (0,0.5) -- (0,2) -- 
										%(-1.5,1.5) -- 
										%(0,0.5);
										\node[vertex] (i1) at  (-4,0) 
										{$X_{t-3}^I$};
										\node[vertex] (i2) at  (-2,0) 
										{$X_{t-2}^I$};
										\node[vertex] (i3) at  (0,0) 
										{$X_{t-1}^I$};
										\node[vertex] (i4) at  (2,0) 
										{$X_{t}^I$};
										\node[vertex] (a1) at  (-4,-2) 
										{$X_{t-3}^A$};
										\node[vertex] (a2) at  (-2,-2) 
										{$X_{t-2}^A$};
										\node[vertex] (a3) at  (0,-2) 
										{$X_{t-1}^A$};
										\node[vertex] (a4) at  (2,-2) 
										{$X_{t}^A$};
										\node[vertex] (b1) at  (-4,-4) 
										{$X_{t-3}^B$};
										\node[vertex] (b2) at  (-2,-4) 
										{$X_{t-2}^B$};
										\node[vertex] (b3) at  (0,-4) 
										{$X_{t-1}^B$};
										\node[vertex] (b4) at  (2,-4) 
										{$X_{t}^B$};
										\node[vertex] (u1) at  (-4,-6) 
										{$X_{t-3}^U$};
										\node[vertex] (u2) at  (-2,-6) 
										{$X_{t-2}^U$};
										\node[vertex] (u3) at  (0,-6) 
										{$X_{t-1}^U$};
										\node[vertex] (u4) at  (2,-6) 
										{$X_{t}^U$};

										\node at (-5.5,0.25) {\Large 
											\textbf{C}};
										
										%edges
										\draw[edge] (i1) to 
										(i2);
										\draw[edge] (i2) to  
										(i3);
										\draw[edge] (i3) to 
										(i4);
										\draw[edge] (a1) to 
										(a2);
										\draw[edge] (a2) to  
										(a3);
										\draw[edge] (a3) to 
										(a4);
										\draw[edge] (b1) to 
										(b2);
										\draw[edge] (b2) to  
										(b3);
										\draw[edge] (b3) to 
										(b4);
										\draw[edge] (u1) to 
										(u2);
										\draw[edge] (u2) to  
										(u3);
										\draw[edge] (u3) to 
										(u4);
										
										\draw[edge] (i1) to 
										(a2);
										\draw[edge] (i2) to  
										(a3);
										\draw[edge] (i3) to 
										(a4);
										\draw[edge] (a1) to 
										(b2);
										\draw[edge] (a2) to  
										(b3);
										\draw[edge] (a3) to 
										(b4);
										\draw[edge] (u1) to 
										(a2);
										\draw[edge] (u2) to  
										(a3);
										\draw[edge] (u3) to 
										(a4);
										\draw[edge] (u1) to 
										(b2);
										\draw[edge] (u2) to  
										(b3);
										\draw[edge] (u3) to 
										(b4);
										\draw[edge] (b1) to 
										(a2);
										\draw[edge] (b2) to  
										(a3);
										\draw[edge] (b3) to 
										(a4);
										
										\end{tikzpicture}
									\end{minipage}
									
									%\end{tabular}
									\caption{Graphical representations of the 
										examples in Section \ref{sec:ivm}. 
										\textbf{A}: Graph representing the IV 
										model 
										in Example \ref{exmp:classicalIV}. Each 
										node $Y \in \{I,A,B,U \}$ 
										represents a random 
										variable in the model. \textbf{B}: 
										Graph 
										representing the time series IV model 
										in 
										Example \ref{exmp:tsIV}. Each node 
										$Y \in \{I,A,B,U \}$ 
										represents 
										a coordinate process, i.e., 
										$(X_t^Y)_{t\in\mathbb{Z}}$. 
										\textbf{C}: An 
										\emph{unrolled} version of 
										\textbf{B}
										representing 
										the time 
										series IV 
										model \citep{danks2013}. Each node 
										represents a random 
										variable. The analogous graph with a 
										node 
										for every random variable in the time 
										series is 
										known 
										as the \emph{full time graph} 
										\citep{peters2013causal}.}
									\label{fig:IVgraphs}
								\end{figure}
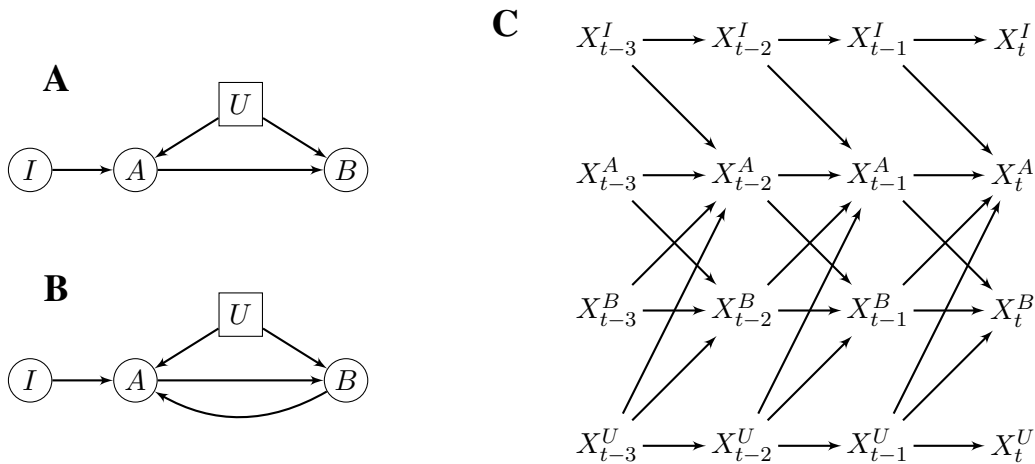

		In this section, we give an example of a classical IV problem, that is, 
		using variables that are not indexed by time. We then compare this to a 
		simple vector-autoregressive model of order 1, VAR(1). In this model, 
		we explain the \emph{variable}- and \emph{process}-centric approaches 
		to 
		IV estimation and show how the integrated covariance enables IV 
		estimation. We assume zero-mean random variables as the generalization 
		is straightforward.
		
				\begin{exmp}[Classical IV]
					\label{exmp:classicalIV}
					Assume we have observable, zero-mean random 
					variables, $I, A, B$, and
					
					$$
					B = \phi A + \varepsilon
					$$
					
					\noindent where $\varepsilon$ is a zero-mean random 
					variable. We wish to estimate $\phi \in \mathbb{R}$. If 
					$\varepsilon$ and $A$ are correlated, then least-squares 
					estimation is biased. If $I$ is 
					uncorrelated with $\varepsilon$ and 
					$E(AI) \neq 0$, then we say that $I$ is an 
					\emph{instrumental variable}. Multiplying by $I$, and 
					taking 
					expectations, we obtain
					
					\begin{align}
						E(BI) = \phi E(AI).
						\label{eq:classical}
					\end{align}
					
					\noindent This moment equation identifies the parameter 
					$\phi$ as $E(AI) 
					\neq 0$. This is also true when $A$ and $\varepsilon$ are 
					correlated, for 
					instance, due to an unobserved confounder, $U$, see Figure 
					\ref{fig:IVgraphs}\textbf{A}. 
				\end{exmp}
				
					From the above it is clear that the parameter $\phi$ is in 
					fact 
					identified from the covariance matrix of the vector 
					$(I,A,B)^T$, that is, the observed covariance matrix is 
					sufficient for IV estimation. The central idea of this 
					paper is to use a different observable matrix in a 
					stochastic 
					process setting which is shown to be sufficient for IV 
					estimation. 
					The next example illustrates this in a simple manner.
		
		\begin{exmp}[Time series IV]
			\label{exmp:tsIV}
					We consider a time series model with a 
					structure which is similar to that in Example 
					\ref{exmp:classicalIV}. Let $X_t = 
					(X_t^I, X_t^A, X_t^B, X_t^U)^T$ such that 
					$X_t^U$ is unobserved and processes $X_t^I, X_t^A, X_t^B, 
					X_t^U$ are all one-dimensional and zero-mean. For 
					simplicity, we assume 
					$X_t$ to 
					be a vector-autoregressive process of order 1 (VAR(1)),
					
					\begin{align*}
						X_t = \Phi X_{t-1} + \varepsilon_t,
					\end{align*}
					
					\noindent where $\varepsilon_t$ are identically 
					distributed and independent random vectors with independent 
					entries. The matrix $\Phi$ has the following structure,
					
					\begin{align*}
						\Phi &= \begin{bmatrix}
							\Phi_{II}       & 0 & 0 & 0 \\
							\Phi_{AI}       & \Phi_{AA} & \Phi_{AB} & \Phi_{AU} 
							\\
							0 & \Phi_{BA} & \Phi_{BB} & \Phi_{BU} \\
							0      & 0 & 0 &\Phi_{UU}
						\end{bmatrix}.
					\end{align*}
					
					We assume that each entry of $\Phi$ is nonzero if it is 
					not explicitly zero above. There is a 
					graphical representation of this process in Figure 
					\ref{fig:IVgraphs}\textbf{B} where $Z \rightarrow Y$ if and 
					only if $\Phi_{YZ} \neq 0$ for $Z,Y \in \{I,A,B,U \}$, 
					$Z\neq Y$. 
					In Figure 
					\ref{fig:IVgraphs}, graph \textbf{C} is an \emph{unrolled} 
					version 
					 of graph \textbf{B} where the nodes 
					represent random variables and $X_{t-1}^Z \rightarrow 
					X_t^Y$ if and only if $\Phi_{YZ} \neq 0$ \citep{danks2013}.
					
					If we were to apply the approach from Example 
					\ref{exmp:classicalIV}, we could consider using $X_{t-2}^I$ 
					as an 
					instrument to identify the parameter $\Phi_{BA}$ which 
					corresponds to the 
					edge 
					$X_{t-1}^A \rightarrow X_{t}^B$ and write
					
					\begin{align*}
						X_t^B &= \Phi_{BA}X_{t-1}^A + \Phi_{BB}X_{t-1}^B + 
						\Phi_{BU}X_{t-1}^U + 
						\varepsilon_t^B \\
						& =  \Phi_{BA}X_{t-1}^A + 
						\bar{\varepsilon}_t^B \\
						E(X_t^BX_{t-2}^I) & = \Phi_{BA} E(X_{t-1}^AX_{t-2}^I) + 
						E(\bar{\varepsilon}_t^BX_{t-2}^I)
					\end{align*}
					
					\noindent where $\bar{\varepsilon}_t^B = \Phi_{BB}X_{t-1}^B 
					+ 
					\Phi_{BU}X_{t-1}^U + 
					\varepsilon_t^B$. 
					\cite{thams2022} (Proposition 6) show that using the moment 
					equation in (\ref{eq:classical}), with $I = X_{t-2}^I, A = 
					X_{t-1}^A, B = X_{t}^B$, does not lead to 
					consistent estimation of $\Phi_{BA}$ when both $\Phi_{II}$ 
					and $\Phi_{BB}$ 
					are nonzero. Therefore, naive application of classical IV 
					methods will not give consistent estimation in this 
					problem. This can be explained by the fact that 
					there are 
					confounding paths going back in time, e.g., $X_{t-2}^I 
					\leftarrow X_{t-3}^I \rightarrow X_{t-2}^A \rightarrow 
					X_{t-1}^B \rightarrow X_{t}^B$, corresponding to the fact 
					that $	E(\bar{\varepsilon}_t^BX_{t-2}^I)$ is not 
					necessarily zero.

					\cite{thams2022} instead provide 
					consistent 
					estimators of $\Phi_{BA}$ using \emph{conditional} 
					instrumental 
					variables, using a conditional version of the moment 
					equation in Equation (\ref{eq:classical}). In this case, 
					$I_{t-2}$ is a conditional instrument for the parameter 
					$\Phi_{BA}$ \emph{conditionally} on $X_{t-3}^I$  
					\citep[Theorem 7]{thams2022}.
		\end{exmp}

		The conditional instrumental variable approach is 
		\emph{variable-centric} in the sense that it identifies finite sets of 
		variables that satisfy assumptions of a conditional instrumental 
		variable method as in the above example. In this 
		paper, we take a different approach which will also provide a 
		solution to the instrumental variable problem above. 
		Instead of looking at covariances of single variables, e.g., 
		between $X_t^B$ and $X_{t-2}^I$, we use an integrated 
		measure of 
		covariance, summing out temporal dependence. Taking this point of view, 
		we arrive at an unconditional 
		instrumental variable method in the above example, and we say that this 
		is a \emph{process-centric} approach as it uses the integrated 
		covariance. The rest of this section describes this idea in the 
		VAR(1)-example, though we no longer require $X_t^I, X_t^A, X_t^B,$ and 
		$X_t^U$ to be one-dimensional.
		
		Assume that the observed variables are mean-zero and that 
		the largest absolute value of the eigenvalues of $\Phi$ is strictly 
		less than one. In the model from Example \ref{exmp:tsIV}, we see that 
		for a fixed $t$, and using that $\varepsilon_{t-j}$ and 
		$\varepsilon_{t+k-l}$ are independent unless $k = l - j $,
		
		\begin{align}
			C &= \sum_{k = -\infty}^\infty E\left(X_t X_{t+k}^T\right) = 
			\sum_{k 
				= -\infty}^\infty E\left(\left(\sum_{j= 0}^\infty 
			\Phi^j\varepsilon_{t-j}\right) \left(\sum_{l= 0}^\infty 
			\Phi^l\varepsilon_{t+k-l}\right)^T\right) \nonumber \\
			&= \left(\sum_{j= 0}^\infty 
			\Phi^j\right) \Theta \left(\sum_{l= 0}^\infty 
			\Phi^l\right)^T = (I -\Phi)^{-1}\Theta (I- \Phi)^{-T}
			\label{eq:intCovIntro}
		\end{align}

		\noindent where $\Theta$ is the diagonal covariance matrix of 
		$\varepsilon_t$. This result also follows from standard VAR-process 
		results \citep[see the covariance matrix generating 
		function]{brockwell2009time}. We 
		will say that Equation (\ref{eq:intCovIntro}) is the \emph{integrated 
		covariance equation}. More 
		general time series models and the linear Hawkes model also satisfy 
		this equation when the 
		parameter matrices are given the correct interpretations. There is also 
		a clear similarity with the parametrization of the observed 
		covariance of a linear structural equation model as noted by 
		\cite{mogensen2022equality}  in the linear 
		Hawkes model. Therefore, more general
		identification results from cyclic linear structural equation models 
		may be used \citep{mogensen2022equality}. 
		
		One can straightforwardly show that $(I - \Phi_{BB})^{-1}\Phi_{BA} = 
		C_{BI}(C_{AI})^{-1}$ when $C_{AI}$ is invertible, thus identifying the 
		matrix $(I - \Phi_{BB})^{-1}\Phi_{BA}$ of \emph{normalized parameters} 
		(Subsection \ref{ssec:norm}). This matrix has a clear causal 
		interpretation, 
		summarizing the direct influence of one subprocess on 
		another (see Subsections \ref{ssec:norm} and \ref{ssec:hawkesNorm}). In 
		the following sections, we show that this approach also 
		applies to 
		more 
		general time series models as well as to 
		linear Hawkes processes, a class of multivariate, continuous-time point 
		processes.

		\section{Probabilistic Models}
		\label{sec:pm}
		
		In this section, we introduce the VAR(p)-models while we defer the 
		introduction of the 
		\emph{linear Hawkes processes} until Section \ref{sec:hawkes}. In both 
		cases, a version of Equation 
		(\ref{eq:intCovIntro}) is satisfied 
		which enables the instrumental variable methods of Section 
		\ref{sec:ip}. This section also describes \emph{normalized parameters} 
		in more 
		detail as these will constitute our estimands. It may seem surprising 
		that the same instrumental variable technique applies to both time 
		series and continuous-time point processes, however, similar parallels 
		have been studied in other work \citep{brillinger1994time}. Despite 
		the strong connections between the two settings, the linear Hawkes 
		processes are only introduced at the end of the paper to obtain a 
		simpler presentation.  We do, however, compare the 
		two cases 
		throughout the paper to highlight the similarities and differences.

		\subsection{Time Series}
		\label{ssec:ts}
				
	Let $X_t = (X_t^1,\ldots,X_t^n)^T$ be a multivariate time series in 
	discrete time, $t \in \mathbb{Z}$. We assume that $X_t$ is stationary and 
	that $E((X_t^i)^2)<\infty$ for all $i$ and $t$. We 
	say that $X_t$ is a 
	VAR(p)-process if
	
	\begin{align}
		X_t &= \sum_{k= 1}^p \Phi_k X_{t-k} + \varepsilon_t
		\label{eq:VARp}
	\end{align}
	
	\noindent where the $\varepsilon$-process is mean-zero and stationary, 
	$\varepsilon_t$ and $\varepsilon_s$ are uncorrelated for $s\neq t$, and 
	$E(\varepsilon_t \varepsilon_t^T) = \Theta$. Define $\Phi(z) = I - \Phi_1 z 
	- \ldots - \Phi_p z^p$. We assume 
	that 
	$\mathrm{det}(\Phi(z))\neq 0$ for all $z \in \mathbb{C}$ such that 
	$\vert z \vert \leq 1$. This means that there exists a unique 
	stationary solution to the VAR(p)-equation \citep[Theorem 
	11.3.1]{brockwell2009time} and we assume throughout that we observe 
	a stationary time series. We use the notation $\Phi = 
	\sum_{i=1}^p 
	\Phi_i$. The above assumption on $\Phi(z)$ implies that $I - 
	\Phi$ is invertible. The entries of the matrix $(I - \Phi)^{-1}$ are 
	sometimes called \emph{long-run effects} \citep{lutkepohl2005new}. We also 
	assume 
	that $I - \Phi_{BB}$ is invertible when needed, see also Subsection 
	\ref{ssec:norm}. The \emph{spectral radius}, $\rho(A)$, of a square matrix 
	$A$ is the largest absolute value of its eigenvalues, and we assume 
	$\rho(\Phi) < 1$.
	
	We define the \emph{integrated 
		covariance} of the time series $X_t$,
	
\begin{align*}
	C = \sum_{k=-\infty}^\infty E(X_t X_{t+k}^T) - E(X_t)E(X_{t+k})^T.
\end{align*}

	\noindent 	The matrix $C$ is well-defined since the sum converges 
	\citep[p. 420]{brockwell2009time}. \cite{brockwell2009time} (Section 11.2) 
	discuss estimation 
	of 
	the terms $E\left(X_t X_{t+i}^T\right)$. The matrix $C$ is independent of 
	$t$ due to 
	stationarity. One should also note that the matrix $C$ equals $2\pi$ times 
	the spectral density matrix of 
	$X_t - E(X_t)$ at 0. 
	
	We saw in Section \ref{sec:ivm} that the integrated 
	covariance equation holds for VAR(1)-processes and we can extend this 
	result to VAR(p)-processes. This is done in Appendix \ref{app:intCovVARp}. 
	We obtain the same integrated covariance equation, however,
		$\Phi$ is now the sum of the direct effects for each lag $k = 
		1,\ldots,p$, that is, $\Phi = \sum_{k = 1}^{p} \Phi_k$. Again, this 
		result is also implied by textbook results on 
		time series 
		\citep[p. 420]{brockwell2009time}.

	We note that $\Phi$ in the VAR(p)-case may have negative entries which is 
	different from the linear Hawkes case. This means that 
	some results that apply in the linear Hawkes setting do not hold in 
	VAR(p)-time series, e.g., in relation to marginalization
	\citep{mogensen2022equality, hyttinen2012learning}.

		\subsection{Normalized Parameters}
		\label{ssec:norm}
		
		The entries of the parameter matrix $\Phi$ have an intuitive 
		interpretation. However, in general we will not 
		be able to identify 
		these parameters with the methods in this paper, see Example 
		\ref{exmp:norm}. 
		Instead, we will aim to identify the entries of the \emph{normalized} 
		parameter matrix. These parameters also have a useful interpretation as 
		we 
		will explain. We use $I_n$ to denote the identity matrix of 
		dimension $n$ and $I_b$ to denote the identity matrix of dimension 
		$\vert B\vert$ for a finite set $B$.
		
		\begin{defn}[Normalized parameters]
			Consider a pair of matrices $(\Phi, \Theta)$ that solve the 
			integrated covariance equation. We say that they are 
			\emph{normalized} if $\Phi_{ii} = 0$ for all $i$.
			\label{def:norm}
		\end{defn}
		
		Say we consider any pair $(\Phi, \Theta)$ and wish to normalize it. We 
		define $D$ to be the 
		diagonal matrix such that $D_{ii} = (1 - \Phi_{ii})^{-1}$.
		Note that $\Phi_{ii} \neq 1$ due to the assumptions on $\Phi$ 
		(Subsection 
		\ref{ssec:ts} and Section \ref{sec:hawkes}). The matrix $D$ is 
		invertible and
		
		\begin{align}
			C  &= (I_n - \Phi)^{-1}\Theta(I_n - \Phi)^{-T} = (D(I_n - 
			\Phi))^{-1}D\Theta D(D(I_n - \Phi))^{-T} \\
			& = (I_n - \bar{\Phi})^{-1}\bar{\Theta}(I_n - \bar{\Phi})^{-T},
			\label{eq:normEq}
		\end{align}
		
		\noindent $\bar{\Phi} = I_n - D(I_n - \Phi)$. We see that 
		$(\bar{\Phi})_{ji} = \Phi_{ji}/(1 - \Phi_{jj})$ for $j\neq i$ 
		and that $\bar{\Phi}$ has zeros on the diagonal and 
		therefore $(\bar{\Phi}, \bar{\Theta})$ is normalized.
		
		In a VAR(1)-model, we see that $(\Phi_{jj})^k\Phi_{ji}$ is the partial 
		causal effect corresponding to the path $X_t^i \rightarrow X_{t+1}^j 
		\rightarrow X_{t+2}^j \rightarrow \ldots \rightarrow X_{t+k + 1}^j$. If 
		$\vert \Phi_{jj}\vert < 1$, we 
		have $\Phi_{ji}/(1 - \Phi_{jj}) = 
		\sum_{k=0}^{\infty}(\Phi_{jj})^k\Phi_{ji}$ and the normalized parameter 
		is therefore the sum of the partial effects \citep{tian2004identifying} 
		along all paths of the type 
		$X_t^i \rightarrow X_{t+1}^j 
		\rightarrow X_{t+2}^j \rightarrow \ldots \rightarrow X_{t+k + 1}^j$ and 
		a measure of the causal influence of the variable $X_t^i$ on the entire 
		future of process $j$, counting the direct effect as well as subsequent 
		self-effects.  The normalized parameters are seen to represent an 
		easily interpretable 
		causal quantity.
		
		We will also use quantities of the type $(I_b - 
		\Phi_{BB})^{-1}\Phi_{BA}$ which is a multivariate version of the above. 
		The interpretation generalizes in a straightforward manner to this 
		case. We see that $\Phi_{BB}^k \Phi_{BA}$ are the partial 
		effects \citep{tian2004identifying} from $X_{t}^A$ 
		to $X_{t+k+1}^B$ corresponding to paths $A 
		\rightarrow B \rightarrow B \rightarrow \ldots \rightarrow B$. If 
		$\rho(\Phi_{BB}) < 1$, this 
		means that $(I_b - \Phi_{BB})^{-1}\Phi_{BA} = \sum_{k = 0}^\infty 
		\Phi_{BB}^k \Phi_{BA}$ is an aggregate causal effect from $X_{t}^A$ to 
		$\{X_{t+j}^B \}_{j \geq 1}$ taking only paths of the type $A 
		\rightarrow B \rightarrow B \rightarrow \ldots \rightarrow B$ into 
		account. In this sense, it is a direct effect of $A$ at time $t$ on the 
		entire future $B$-process counting the direct effect $X_t^A \rightarrow 
		X_{t+1}^B$ and subsequent self-effects within $B$. Therefore, this is a 
		natural quantification of the effect of subprocess $A$ on subprocess 
		$B$ when taking 
		a 
		stochastic process point of view. The condition that $\rho(\Phi_{BB}) < 
		1$ is a natural requirement for the above interpretation of normalized 
		parameters as this means that 
		the marginal $B$-time series is 
		`stable' in itself. This ensures that the induced self-effects are 
		finite.
		
Example \ref{exmp:norm} in Appendix \ref{app:norm} shows that from a normalized 
pair, $(\Phi, \Theta)$, we can find different pairs $(\bar{\Phi}, 
\bar{\Theta})$ solving the same integrated 
covariance equation as the original pair. If $\rho(\Phi) < 1$ and the entries 
are nonnegative and we let $0 < D_{ii} < 1$, then the same holds for 
$\bar{\Phi}$. This means that in both the time series case and the linear 
Hawkes 
case we may find infinitely many pairs $(\bar{\Phi},\bar{\Theta})$ that solve 
the equation. In the time series 
case, we need to argue that $I - \bar{\Phi}_{BB}$ is also invertible. These 
arguments are provided in Example \ref{exmp:norm} in Appendix \ref{app:norm}. 
\cite{hyttinen2012learning} provide similar arguments in the context of cyclic 
linear structural equation models.

\subsection{Graphical Representation}

	One may use graphs to represent 
	assumptions that are sufficient for IV analysis. These graphs are defined 
	for VAR(p)-models 
	below.

\begin{defn}[Causal graph]
	Let $\mathcal{G}$ be a directed graph on nodes $V$ and with edge set $E$. 
	In the VAR(p)-model, we say that $\mathcal{G}$ is the 
	\emph{causal graph} 
	 of the process if $i \rightarrow j$ is in $E$, $i\neq j$, if and only if
	 there exists $k$ such that $(\Phi_k)_{ji} \neq 0$.
	 \label{def:causalGraph}
\end{defn}

Note that the causal graph does not contain loops, that is, edges $i 
\rightarrow i$. When identifying \emph{normalized} parameters, loops are 
inconsequential as the normalization removes self-effects and adjusts the other 
parameters to retain the integrated covariance.

 We say that $i$ is a \emph{parent of} $j$ if 
 $i\rightarrow j$ in $\mathcal{G}$, or if $i=j$, and we say 
 that a process, $j$, is 
 \emph{exogenous} if it has no parents other than $j$ in the causal 
graph. 	We say that a subset of processes, $I\subseteq V$, are 
\emph{exogenous} 
if there are no $i\notin I$ and $j\in I$ such that 
$i\rightarrow j$ in the causal graph. Note that there could be 
edges between processes in an exogeneous set, $I$, only not from 
processes $V\setminus I$ and into $I$.
		
		\section{Instrumental Processes}
		\label{sec:ip}
		
		We first describe the results for VAR(p)-process. The linear Hawkes 
		setting is analogous as we will see in Section \ref{sec:hawkes}, 
		though, 
		the interpretation of the parameters differs. Section 
		\ref{sec:confoundingts} provides a result for more general time series.
		
		This section uses the algebraic equation in (\ref{eq:intCovIntro}) to 
		define instrumental processes that allow us to identify 
		\emph{normalized} causal parameters (see Definition \ref{def:norm}). 
		\cite{mogensen2022equality} notes that the 
		parametrization of the integrated covariance is similar to the 
		parametrization of the covariance of a linear structural equation 
		model for which there are several identification results, see, e.g., 
		\cite{foygelHalftrek2012, chenNIPS2016, weihs2018}. 
		We will not use this 
		connection directly and therefore we refer to that paper for a detailed 
		explanation. One should note that identification results from 
		linear SEMs could be 
		used to obtain some of the results of this paper. However, we take a 
		more direct approach which is closer to other IV work. Furthermore, 
		this approach also makes the needed assumptions explicit whereas 
		identification results are often only generic, that is, hold 
		outside a 
		measure-zero set of parameters. 
		
		The results in this section are similar in spirit 
		to other IV work, however, we use the matrix $C$ directly, and not a 
		set of 
		random variables. $C$ is easily seen to be similar to a covariance 
		matrix, but it is an aggregate measure of covariance between processes 
		rather than the covariance of a set of observed random variables.

		We first give a univariate definition of an instrumental process which 
		leads to an identification result. We then define a multivariate 
		instrumental process and state the corresponding identification result. 
		The 
		univariate definition and result are naturally implied by the 
		multivariate result. However, we include them in this order to present 
		the 
		simplest possible setting first. The 
		symbol $\iota$ will throughout the paper denote an instrumental process 
		for the effect from $\alpha$ to $\beta$, where $\iota,\alpha,\beta\in 
		O$ 
		and 
		$O\subseteq V$ is a set of observed processes. The symbol $I$ 
		will denote an instrumental process which is instrumental for 
		the effect from the set $A$ to the set $B$, where $I,A,B \subseteq 
		O$. We assume that $I,A,$ 
		and $B$ are disjoint. We say that $\alpha$ is a 
		\emph{descendant of} $\beta$ in $\mathcal{G}$ if there exists a 
		directed 
		path $\beta\rightarrow \ldots \rightarrow \alpha$. We let
		$\mathrm{de}_\mathcal{G}(\beta)$ denote the set of descendants of 
		$\beta$. We let 
		$\mathrm{pa}_\mathcal{G}(\beta)$ denote the set of parents of $\beta$.
		By convention, $\beta\in\mathrm{de}_\mathcal{G}(\beta)$ and 
		$\beta\in\mathrm{pa}_\mathcal{G}(\beta)$. 
		We define $\mathrm{de}_\mathcal{G}(B) = \cup_{\beta\in 
			B}\mathrm{de}_\mathcal{G}(\beta)$ and $\mathrm{pa}_\mathcal{G}(B) = 
			\cup_{\beta\in 
			B}\mathrm{pa}_\mathcal{G}(\beta)$.
		
		\begin{defn}
			Let $\iota, \alpha,$ and $\beta$ be distinct. We say that $\iota$ 
			is a
			\emph{VAR(p)-instrumental process} for $\alpha\rightarrow\beta$ in 
			the 
			causal graph
			$\mathcal{G}$ if $\iota$ is exogenous, 
			$\mathrm{de}_\mathcal{G}(\iota)\cap \mathrm{pa}_\mathcal{G}(\beta) 
			\subseteq \{\alpha,\beta \}$, and 
			$C_{\alpha\iota} \neq 0$.
			\label{def:uniIV}
		\end{defn}
	
	We will later give a more general definition of an instrumental process, 
	and we say \emph{instrumental process} instead of \emph{VAR(p)-instrumental 
	process} when it is clear from the context that we are considering a 
	VAR(p)-model. We emphasize that $\iota \in V$ is a \emph{coordinate 
	process}. In the time 
	series case, it is the collection of random variables $(X_t^\iota)_{t\in 
	\mathbb{Z}}$. \cite{jiang2023instrumental} describe an instrumental 
	variable method 
	in point process models using a random variable as an 
	instrument. In contrast, we are using an entire coordinate process as 
	an instrument.
	
	One should also note that Definition \ref{def:uniIV} makes assumptions on 
	the causal graph from Definition \ref{def:causalGraph}. This graph is 
	constructed from the parameters of the time series and 
	therefore Definition \ref{def:uniIV} also imposes restrictions on the way 
	the coordinate processes interact.

		\begin{thm}
			Let $\mathcal{G}= (V,E)$ 
			be a causal graph, $V=O\disjU U$, and let $\iota, \alpha,\beta \in 
			O$. If $\iota$ is an instrumental 
			process for $\alpha\rightarrow\beta$, then 
			$(1 - 
			\Phi_{\beta\beta})^{-1}\Phi_{\beta\alpha}$ is 
			identified by $C_{\beta\iota}/C_{\alpha\iota}$.
			\label{thm:uniIV}
		\end{thm}

		\begin{exmp}[Instrumental process]
			In this example, we show that the classical IV graph also allows 
			an IV analysis in this setting. Say we have a four-dimensional 
			VAR(1)-process such that the causal graph is as shown in 
			Figure \ref{fig:IVexample}\textbf{A} and process 4 is 
			unobserved. Process 1 is an 
			instrument for $2\rightarrow 3$. Theorem 
			\ref{thm:uniIV} 
			gives that 
			
			$$
			C_{3,1}/C_{2,1}
			$$
			
			\noindent identifies the normalized effect $(1 - 
			\Phi_{33})^{-1}\Phi_{32}$. Section \ref{app:var2exmp} provides an 
			example using a VAR(2)-model and 
			its unrolled graph.
			
						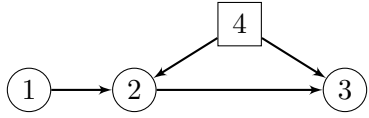
\begin{figure}
							%\begin{tabular}{cc}
							\begin{minipage}{1\linewidth}
								\centering
								\begin{tikzpicture}[scale=0.7]
								\tikzset{vertex/.style = 
								{shape=circle,draw,minimum 
										size=1.5em, 
										inner 
										sep = 0pt}}
								\tikzset{edge/.style = {->,> = latex', thick}}
								\tikzset{edgebi/.style = {<->,> = latex', 
								thick}}
								\tikzset{every loop/.style={min distance=8mm, 
								looseness=5}}
								\tikzset{vertexFac/.style = 
								{shape=rectangle,draw,minimum 
										size=1.5em, 
										inner sep = 0pt}}
								
								% vertices
								%\draw [line width=35pt,opacity=0.1, blue,line 
								%cap=round,rounded
								%corners] (0,0.5) -- (0,2) -- (-1.5,1.5) -- 
								%(0,0.5);
								\node[vertex] (a) at  (-4,0) {$1$};
								\node[vertex] (b) at  (-2,0) {$2$};
								\node[vertex] (c) at  (2,0) {$3$};
								\node[vertexFac] (d) at  (0,1.25) {$4$};

								%edges
								
								\draw[edge] (a) to (b);
								\draw[edge] (b) to (c);
								\draw[edge] (d) to (c);
								\draw[edge] (d) to (b);
								
								\end{tikzpicture}
							\end{minipage}
							%\end{tabular}
							\caption{Instrumental process example. Each node 
							$\alpha$, 
							$\alpha \in \{1,2,3,4 \}$, represents an entire 
							\emph{coordinate process}, that is, the collection 
							of random variables 
							$(X_t^\alpha)_{t\in\mathbb{Z}}$ in the time series 
							case. Process 4 is 
							unobserved (indicated by the square). Process 1 
							($\iota$) may 
							serve as an instrumental process to estimate the 
							normalized effect from 2 ($\alpha$) to 3 ($\beta$).}
						\label{fig:IVexample}
						\end{figure}
						
		\end{exmp}

		\subsection{Multiple Instruments}
		
		As in other instrumental variable frameworks, we may consider using 
		\emph{multiple instruments} when there are multiple processes that 
		are instrumental for the same effects. Note that we throughout assume 
		$I_b - \Phi_{BB}$ to be invertible.

		\begin{defn}
			Let $I,A,B \subseteq O$ be disjoint and non-empty sets. We say that 
			a set of processes, $I$, is a 
			\emph{VAR(p)-instrumental process}
			for $A\rightarrow B$ in the causal graph $\mathcal{G}$ if $I$ is 
			exogenous, $\mathrm{de}_\mathcal{G}(I)\cap 
			\mathrm{pa}_\mathcal{G}(B) \subseteq A\cup B$, and 
			$C_{AI}$ has full row 
			rank.
			
			\label{def:multiIV}
		\end{defn}
		
		\begin{thm}[Multiple instruments (just identified)]
			Let $I,A,B\subseteq O$. If $I$ 
			is an instrumental process for the effect $A\rightarrow B$ and 
			$\vert A\vert = \vert I \vert$, then 
			$(I_b - \Phi_{BB})^{-1}\Phi_{BA}$ is identified by 
			$C_{BI}(C_{AI})^{-1}$.
			\label{thm:multiIV}
		\end{thm}
		
		If the condition $\mathrm{de}_\mathcal{G}(I)\cap 
		\mathrm{pa}_\mathcal{G}(B) \subseteq A\cup B$ is not satisfied, one may 
		in some cases choose a larger $B$ to find an instrumental process. 
		Figure \ref{fig:multiIVexample} gives an 
		example 
		of a graphical 
		structure with a multivariate instrumental process.

		\subsection{Overidentification}
		
		Consider instead the case where $\vert A\vert < \vert I\vert$, that is, 
		overidentification.  In this case, $C_{AI}$ is not invertible. Let 
		$C_{AI}^-$ be a right inverse, that is, 
		$C_{AI}^-$ is an $\vert I \vert \times \vert A \vert$ matrix such that 
		$C_{AI}C_{AI}^- = I_{a}$.
		Such a matrix exists as $C_{AI}$ has full row rank by assumption. Note 
		that from this assumption it also follows that $R_{AI}$ 
		has full row rank when $R=(I_n - \Phi)^{-1}$ as $\mathrm{rank}(AB) \leq 
		\mathrm{rank}(A)$ for 
		matrices $A$ and $B$. We see that $R_{AI}^- = 
		\Theta_{II}R_{II}C_{AI}^-$ is a right inverse of $R_{AI}$. We have
		
\begin{align*}
			R_{BI}R_{AI}^- = C_{BI}C_{AI}^-.
\end{align*}
		
		\noindent The proof of Theorem \ref{thm:multiIV} holds also 
		in this case, 
		showing that any choice of right inverse of $C_{AI}$ leads to 
		identification of the normalized parameters. Note that choosing a 
		specific right inverse 
		of $C_{AI}$ specifies a choice of right inverse of $R_{AI}$ as well
		-- this specific right inverse is then used throughout the proof.
		
		When $W$ is a positive definite weight matrix
		then $C_{AI}WC_{AI}^T$ is 
		invertible using the fact that $C_{AI}$ has full rank. We see that the 
		matrix $WC_{AI}^T (C_{AI}WC_{AI}^T)^{-1}$ is 
		a 
		right inverse of $C_{AI}$. This motivates using 
		
			$$
			C_{BI}WC_{AI}^T (C_{AI}WC_{AI}^T)^{-1}
			$$
			
	   \noindent as an estimate in the overidentified setting by plugging in 
	   estimated 
	   entries of $C$. See also \cite{thams2022} and \cite{hall2005generalized}.

								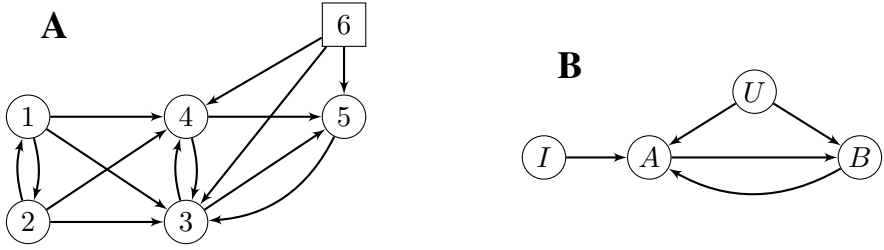
\begin{figure}
									%\begin{tabular}{cc}
									\begin{minipage}{0.4\linewidth}
										\centering
										\begin{tikzpicture}[scale=0.7]
										\tikzset{vertex/.style = 
											{shape=circle,draw,minimum 
												size=1.5em, 
												inner 
												sep = 0pt}}
										\tikzset{edge/.style = {->,> = latex', 
										thick}}
										\tikzset{edgebi/.style = {<->,> = 
										latex', 
												thick}}
										\tikzset{every loop/.style={min 
										distance=8mm, 
												looseness=5}}
										\tikzset{vertexFac/.style = 
											{shape=rectangle,draw,minimum 
												size=1.5em, 
												inner sep = 0pt}}
										
										% vertices
										%\draw [line width=35pt,opacity=0.1, 
										%blue,line 
										%cap=round,rounded
										%corners] (0,0.5) -- (0,2) -- 
										%(-1.5,1.5) -- 
										%(0,0.5);
										\node[vertex] (a) at  (-4,0) {$1$};
										\node[vertex] (b) at  (-4,-2) {$2$};
										\node[vertex] (c) at  (-1,-2) {$3$};
										\node[vertex] (d) at  (-1,0) {$4$};
										\node[vertex] (e) at  (2,0) {$5$};
										\node[vertexFac] (f) at  (2,1.75) {$6$};
										
										\node at (-3.5,1.75) {\Large 
										\textbf{A}};
										
										%edges
										
										\draw[edge] (a) to [bend left = 15] (b);
										\draw[edge] (b) to [bend left = 15] (a);
										\draw[edge] (a) to (c);
										\draw[edge] (b) to (c);
										\draw[edge] (a) to (d);
										\draw[edge] (b) to (d);
										\draw[edge] (d) to [bend left = 15] (c);
										\draw[edge] (c) to [bend left = 15] (d);
										\draw[edge] (c) to (e);
										\draw[edge] (d) to (e);
										\draw[edge] (f) to (e);
										\draw[edge] (f) to (c);
										\draw[edge] (f) to (d);
										\draw[edge] (e) to [bend left] (c);
										
										\end{tikzpicture}
									\end{minipage}\hspace{.05\linewidth}%
									\begin{minipage}{0.4\linewidth}
										\centering
										\begin{tikzpicture}[scale=0.7]
										\tikzset{vertex/.style = 
											{shape=circle,draw,minimum 
												size=1.5em, 
												inner 
												sep = 0pt}}
										\tikzset{edge/.style = {->,> = latex', 
										thick}}
										\tikzset{edgebi/.style = {<->,> = 
										latex', 
												thick}}
										\tikzset{every loop/.style={min 
										distance=8mm, 
												looseness=5}}
										\tikzset{vertexFac/.style = 
											{shape=rectangle,draw,minimum 
												size=1.5em, 
												inner sep = 0pt}}
										
										% vertices
										%\draw [line width=35pt,opacity=0.1, 
										%blue,line 
										%cap=round,rounded
										%corners] (0,0.5) -- (0,2) -- 
										%(-1.5,1.5) -- 
										%(0,0.5);
										\node[vertex] (a) at  (-4,0) {$I$};
										\node[vertex] (b) at  (-2,0) {$A$};
										\node[vertex] (c) at  (2,0) {$B$};
										\node[vertex] (d) at  
										(0,1.25) 
										{$U$};
										
										\node at (-3.5,1.75) {\Large 
										\textbf{B}};
										
										%edges
										
										\draw[edge] (a) to (b);
										\draw[edge] (b) to (c);
										\draw[edge] (c) to [bend left] (b);
										\draw[edge] (d) to (c);
										\draw[edge] (d) to 
										(b);						
										\end{tikzpicture}
									\end{minipage}
									%\end{tabular}
									\caption{ \textbf{A}: 
										Multivariate instrumental process 
										example. Process 6 
										is 
										unobserved (indicated by the square). 
										Processes 1 and 2 
										($I$) may 
										serve as a multivariate instrumental 
										process to 
										estimate the 
										normalized effect from 3 and 4 ($A$) to 
										5 ($B$). 
										\textbf{B}: This graph is a simplified 
										version of \textbf{A}. We collapse 
										processes 1 
										and 2 into a single node and processes 
										3 and 4 into another node, defining 
										sets 
										$I = \{1,2\}, A=\{3,4\}, B= \{5\}, U 
										=\{6\}$ where $U$ is unobserved. For 
										$X,Y \in 
										\{I,A,B,U\}$, $X\neq Y$, we include 
										edges $X \rightarrow Y$ if and only if 
										$x \rightarrow y$ for some $x \in X$ 
										and $y\in Y$. This recovers the 
										`univariate' IV structure from Figure 
										\ref{fig:IVgraphs}\textbf{B}. 
										\cite{thams2022} use this graphical 
										representation as well as the full time 
										graphs as described below Figure 
										\ref{fig:IVgraphs}.}
									\label{fig:multiIVexample}
								\end{figure}
												
		%%%%%%%%%%%%%%%%%%%%%%%%%%%%%%%%%%%%%%
		
		\section{Time Series, General Case}
		\label{sec:confoundingts}
		
		We now argue that the above methods apply to time series models 
		under 
		far more general 
		assumptions than those of a VAR(p)-model. It is also in this case 
		possible to use a graph to 
		represent the assumptions we need for the instrumental variable method, 
		however, we dispense with the graphical conditions in this section. We 
		assume that $(X_t^I, X_t^A, X_t^B, X_t^U)^T$ is a stationary process 
		such that
		$X_t^I, X_t^A, X_t^B$ are mean-zero and
		
		\begin{align}
			X_t^B &= \sum_{j = 1}^{p}\Phi_{j,BA}X_{t-j}^A + \sum_{j = 
			1}^{p}\Phi_{j,BB}X_{t-j}^B + g_B(\ldots, 
			X_{t-2}^U, X_{t-1}^U, 
			\varepsilon_{t}^B) 
			\label{eq:genTimeSeries}
		\end{align}
		
		\noindent where $\Phi_{j,BA}$ and $\Phi_{j,BB}$ are matrices of the 
		appropriate dimensions and $g_B$ is a function. We define 
		$\Phi_{BA} = \sum_{j = 1}^{p}\Phi_{j,BA}$ and $\Phi_{BB} = \sum_{j = 
		1}^{p}\Phi_{j,BB}$. If Equation 
		(\ref{eq:genTimeSeries}) holds, $X_t^I$ is independent of $X^U$ and 
		$\varepsilon^B$ for all 
		$t$, $(I_b - \Phi_{BB})$ is invertible, $C_{BI}$ and $C_{AI}$ are 
		well-defined, and $C_{AI}$ has full rank, then we say that $I$ is an 
		\emph{instrumental process} for $A\rightarrow B$. 
		
		If $(X_t^I, X_t^A, X_t^B, X_t^U)^T$ is a VAR(p)-model (under the 
		stationarity condition of Subsection \ref{ssec:ts}) and Definition 
		\ref{def:multiIV} is satisfied, then $I$ is also an instrumental proces 
		for $A\rightarrow B$ using the above definition, and we see that the 
		above assumptions 
		are less restrictive than those used in the VAR(p)-setting. First, the 
		linearity is only imposed by Equation (\ref{eq:genTimeSeries}) while 
		$I$ and $U$ may depend nonlinearly on their own lagged values as no 
		explicit assumptions are made on their dynamics. Second, 
		the dependence of $A$ on $I$, $U$, and $B$ may also be nonlinear. Using 
		the above definition of an instrumental process, we obtain the next 
		theorem. $(C_{AI})^-$ denotes a right inverse of $C_{AI}$.
		
		\begin{thm}
			If $I$ is an 
			instrumental 
			process for $A\rightarrow B$, then $(I_b - 
			\Phi_{BB})^{-1}\Phi_{BA}$ is identified by $C_{BI}(C_{AI})^{-}$.
			\label{thm:genTimeSeries}
		\end{thm}
		
		\begin{proof}
			 We prove the case $p = 1$ while the general case is in Section 
			 \ref{app:proofs}. We may write
			 
			 \begin{align*}
			 	E\Big(X_{t}^B(X_{t+k}^I)^T\Big) &= E\Big((\Phi_{BA}X_{t-1}^A
			 	+ 
			 	\Phi_{BB}X_{t-1}^B + g_B(\ldots, X_{t-2}^U, X_{t-1}^U, 
			 	\varepsilon_t^B))(X_{t+k}^I)^T\Big) \\
			 	&= \Phi_{BA}E\Big(X_{t-1}^A(X_{t+k}^I)^T \Big) + 
			 	\Phi_{BB}E\Big(X_{t-1}^B(X_{t+k}^I)^T \Big).
			 \end{align*}
			 
			 \noindent We sum over $k$ in the above expression,

			 \begin{align*}
			 	\sum_{k = -\infty}^\infty E\Big(X_{t}^B(X_{t+k}^I)^T\Big)
			 	&= \sum_{k = 
			 	-\infty}^\infty\Phi_{BA}E\Big(X_{t-1}^A(X_{t+k}^I)^T 
			 	\Big) + \sum_{k = -\infty}^\infty
			 	\Phi_{BB}E\Big(X_{t-1}^B(X_{t+k}^I)^T \Big) \\
			 	C_{BI} &= \Phi_{BA} C_{AI} + \Phi_{BB}C_{BI}
			 \end{align*}
			 
			 \noindent and isolating $C_{BI}$ we obtain $
			 C_{BI} = (I_b - \Phi_{BB})^{-1}\Phi_{BA}C_{AI}$. If $C_{AI}$ is 
			 has 
			 full row rank this gives identification of 
			 the 
			 matrix $(I_b - 
			 \Phi_{BB})^{-1}\Phi_{BA}$.
			 
		\end{proof}
		
		One 
		can use conditions such as 
		those in the VAR(p)-models to ensure that the relevant entries of $C$ 
		are well-defined.

		\section{Estimation}
		\label{sec:est}
		
		In order to use the instrumental process framework above, one can 
		estimate the relevant entries of the integrated covariance matrix and 
		then plug in the estimated covariances to obtain estimates of the 
		normalized parameters directly from the identifying formulas. To 
		estimate the matrix $C$ in the time series 
		case, one may 
		use the relation to the spectral density matrix of the time series, 
		see, 
		e.g., \cite{hansen1982large, brillinger2001time}. One may also use the 
		connection to 
		\emph{long-run covariance} to estimate $C$, see, 
		e.g., \cite{andrews1991heteroskedasticity, andrews1992improved}.

		Section \ref{app:numerical} contains numerical examples of the 
		instrumental process method in this paper. Subsection 
		\ref{ssec:hawkesNorm} describes estimation in the linear Hawkes case.
		
		%%%%%%%%%%%%%%%%%%%%%%%%%%%%%%%%%%%%%%

\section{Linear Hawkes Processes}
\label{sec:hawkes}

		A \emph{linear Hawkes 
			process} is a certain kind of \emph{point process}. We give a short 
		introduction here, see also, e.g., \cite{laub2015hawkes, 
			daley2003introduction}. We consider a 
		filtered probability space $(\Omega, \mathcal{F}, (\mathcal{F}_t), P)$ 
		where $(\mathcal{F}_t)$ is a filtration and an index set $V = 
		\{1,2,\ldots, n\}$. For $i\in V$, there is a sequence of random event 
		times 
		$\{T_k^i \}_{k\in \mathbb{Z}}$ such that $T_k^i < T_{k+1}^i$ almost 
		surely. We define a counting process $N_t^i$ such that $N_t^i 
		- N_s^i = \sum_{k} \mathds{1}_{s < T_k^i \leq t}$. 
		Furthermore, we assume that two events cannot occur simultaneously in 
		the multivariate point process. A linear Hawkes process can be defined 
		by imposing constraints on the \emph{conditional intensities}, 
		$\lambda_t^i$. These are stochastic processes and satisfy
		
		\begin{align*}
			\lambda_t^i = \lim_{h\downarrow 0} \frac{1}{h} P(N_{t+h}^i - N_t^i 
			= 1 \mid 
			\mathcal{F}_t)
		\end{align*}
		
		\noindent where $\mathcal{F}_t$ represents the history of the process 
		until time point $t$. A multivariate \emph{linear Hawkes process} is a 
		point process such that 
		
		\begin{align*}
			\lambda_t^j = \mu_j + \sum_{i = 1}^{n} \int_{-\infty}^{t} 
			\phi_{ji}(t - s) dN_s^i
		\end{align*}
		
		\noindent for a nonnegative constant $\mu_j$ and nonnegative functions 
		$\phi_{ji}$ which are zero outside $(0,\infty)$. We assume $\mu_j > 0$. 
		We define $\Phi$ to be 
		the $n \times n$ matrix such that 
		$\Phi_{ji} = \int_{-\infty}^{\infty} \phi_{ji}(s) ds$. See Figure 
		\ref{fig:hawkesIllu} for an illustration of data from a linear 
		Hawkes process. When $A$ is a square matrix, we let $\rho(A)$ denote 
		its \emph{spectral 
			radius}, that is, the largest absolute value of its eigenvalues. We 
		assume that $\rho(\Phi) < 1$ in which case we can assume the linear 
		Hawkes process to have stationary increments \citep{jovanovic2015}. We 
		define the integrated covariance in this setting,
		
		\begin{figure}
			%\begin{tabular}{cc}
			\begin{minipage}{0.48\linewidth}
				\includegraphics[width = \textwidth]{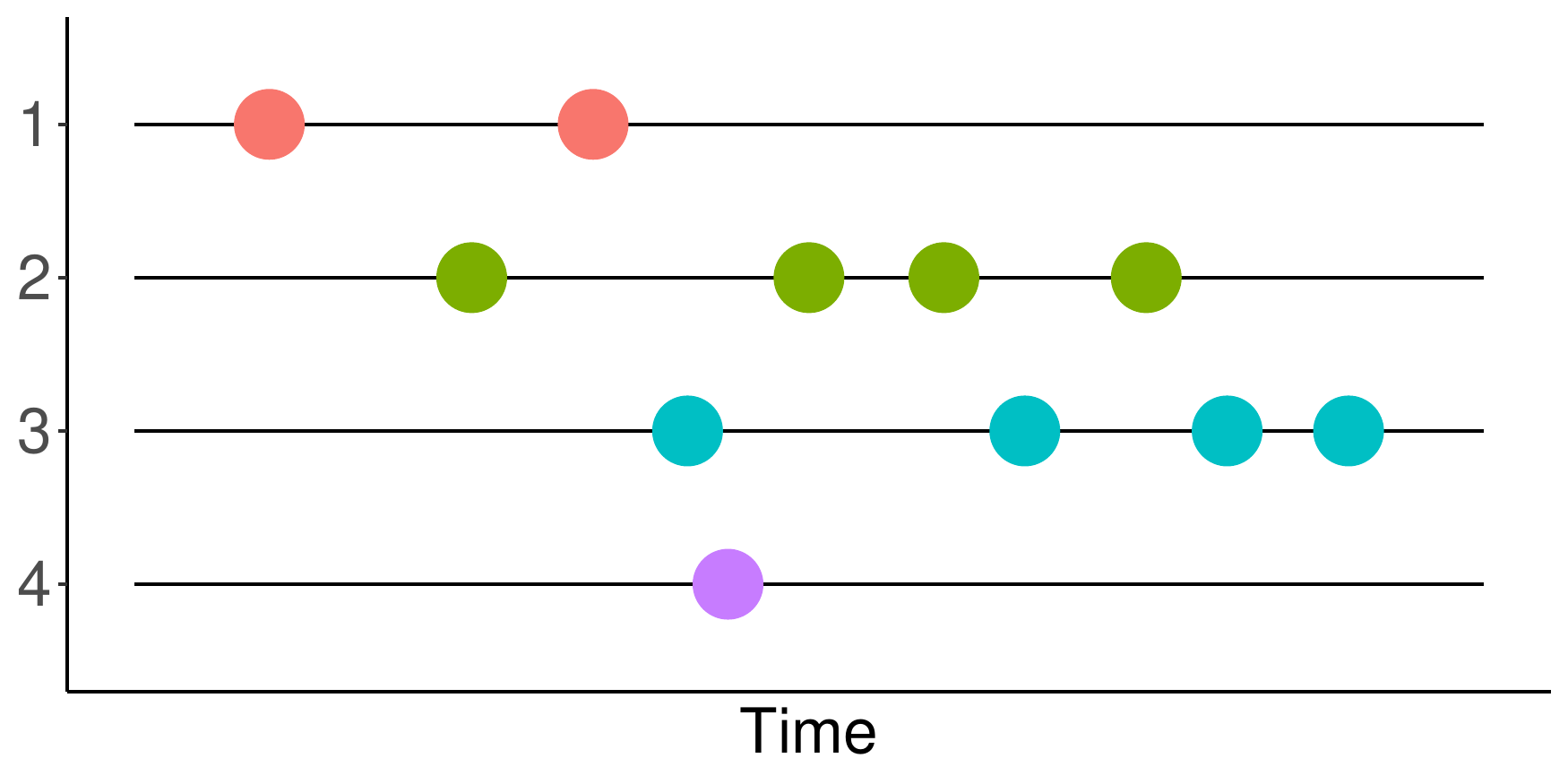}
			\end{minipage}\hspace{.05\linewidth}%
			\begin{minipage}{0.48\linewidth}
				\includegraphics[width = \textwidth]{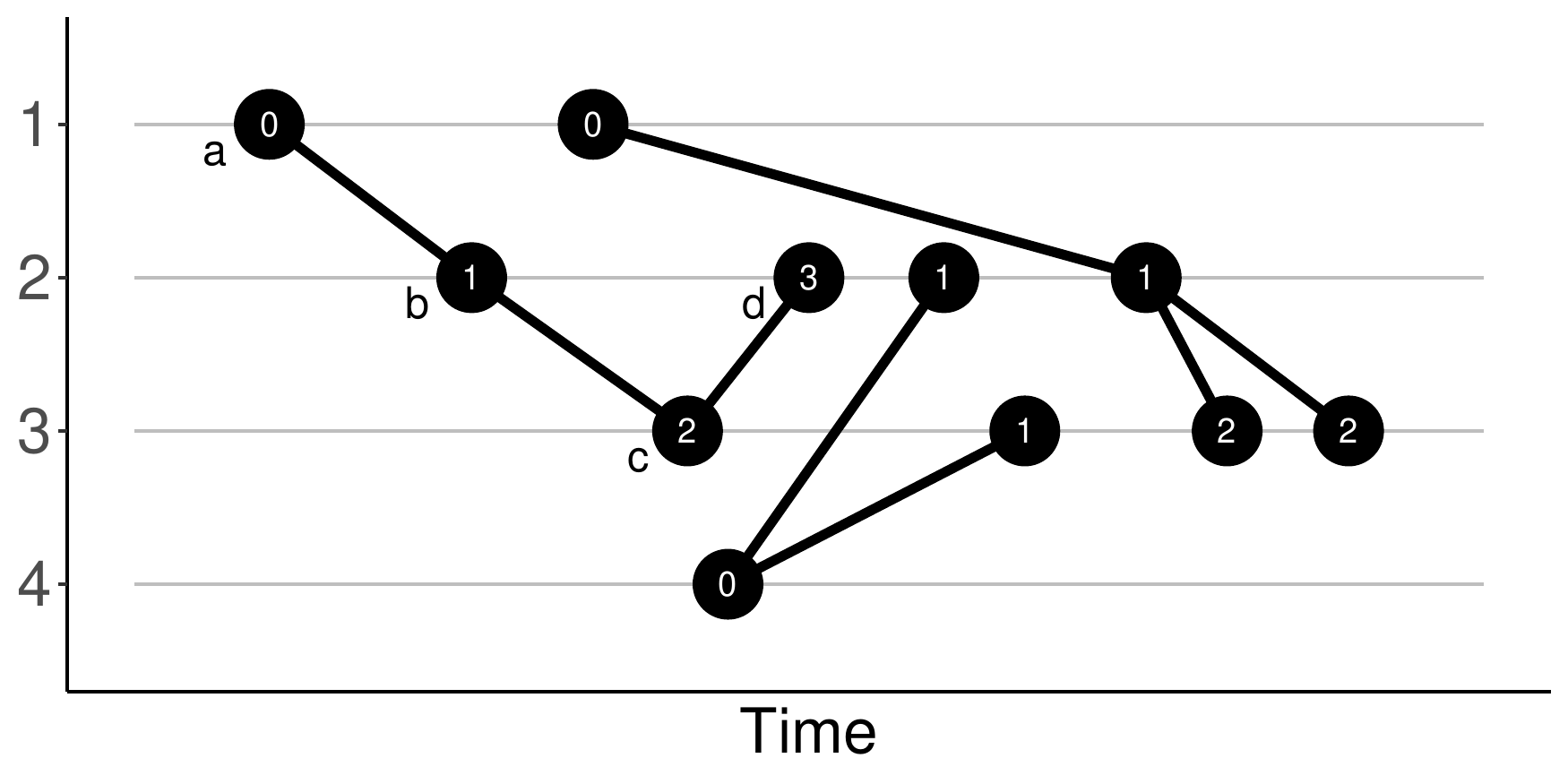}
			\end{minipage}
			
			%\end{tabular}
			\caption{Example data from a 
				four-dimensional linear Hawkes process. 
				Left: Example observed data. Color and 
				vertical placement indicate coordinate 
				process (1, 2, 3, or 4) of the event. 
				Horizontal placement indicates time of the 
				event. Right: The linear Hawkes process can 
				be generated as a \emph{cluster process} 
				where each event may spark future child 
				events, indicated here with line segments. 
				These parent-child relations are 
				unobserved. In the cluster with labeled 
				events ($a, b, c, d$), event 
				$b$ is in the first generation after event 
				$a$ while event $d$ is in the third 
				generation after event $a$. We say that
				$b$ is a \emph{child} of $a$ 
				(direct descendant).}
			\label{fig:hawkesIllu}
		\end{figure}

		\begin{align}
			C_{ij} dt = \int_{-\infty}^{\infty} E(dN_{t}^i dN_{t+s}^j) - 
			E(dN_t^i)E(dN_{t+s}^j) d s.
			\label{eq:hawkesC}
		\end{align}
		
		\noindent We define $\Lambda_i dt = E(d N_t^i)$ and let $\Theta$ 
		denote the diagonal matrix such that $\Theta_{ii} = \Lambda_i$. It 
		holds 
		that 
		
		\begin{align}
			C = (I_n - \Phi)^{-1}\Theta (I_n - \Phi)^{-T},
			\label{eq:Chawkes}
		\end{align}
		
		\noindent see  \cite{achab2017}. This is 
		the same equation 
		as in the VAR(1)-case in Section 
		\ref{sec:ivm}, even 
		though interpretations of the parameter matrices $\Phi$ and $\Theta$ 
		differ. The following definition is analogous to Definition 
		\ref{def:causalGraph}.	

\begin{defn}[Causal graph]
	Let $\mathcal{G}$ be a directed graph on nodes $V$ and with edge set $E$. 
	In the linear Hawkes case, we say that $\mathcal{G}$ is the \emph{causal 
		graph} of the process if $i \rightarrow j$ is in $E$, $i\neq j$, if and 
	only if 
	$\Phi_{ji} \neq 0$.
	\label{def:hawkesCausalGraph}
\end{defn}
		
		\paragraph{Cluster Interpretation}
		
		Above we introduced the linear Hawkes process as a point process with 
		conditional intensities of a certain type. It is, however, possible to 
		give 
		an equivalent definition using the so-called \emph{cluster 
		representation} 
		\citep{jovanovic2015}. 
		We will give a very short description here. For each $i \in V$, a set 
		of 
		generation-0 events are generated from a homogeneous Poisson process 
		with 
		rate $\mu_i$. Each of these events create a \emph{Hawkes cluster} which 
		is 
		generated in the following way. From a generation-$k$ event of 
		type $i$ at time 
		$s$ (coordinate process $i$), generation-($k+1$) events of type 
		$j$ 
		are generated from an inhomogenous Poisson process started at $s$ with 
		rate 
		$\phi_{ji}(t-s)$, $t > s$. This construction is repeated. The 
		superposition 
		of all clusters form a linear Hawkes process. Note that only event 
		types 
		and 
		time points are observed while generation and parent-child relations of 
		an 
		event are unknown when observing data from a linear Hawkes process.
		
		The cluster interpretation provides a straightforward 
		interpretation 
		of the entries of $\Phi$. The entry $\Phi_{ji}$ is the expected number 
		of 
		direct $j$-children from an $i$-event. In general, $(\Phi^k)_{ji}$ is 
		the 
		expected number of $j$-events from an $i$-event in the $k$'th 
		generation 
		from 
		the $i$-event. We define $R = (I_n - \Phi)^{-1} = \sum_{k = 0}^{\infty} 
		\Phi^k$. $R_{ji}$ is the expected total number of $j$-descendants on a 
		cluster 
		rooted 
		at an $i$-event.  The infinite 
		sum 
		converges and $R$ is well-defined due to the assumption on the spectral 
		radius of $\Phi$ 
		\citep{jovanovic2015}.
		See Figure \ref{fig:hawkesIllu} for an
		example of (in)direct descendant events.

\subsection{Normalized Parameters}
\label{ssec:hawkesNorm}

We can normalize the parameters of Equation (\ref{eq:Chawkes}) just as we did 
in Equation (\ref{eq:normEq}). In the 
linear Hawkes process, the normalized parameter $\bar{\Phi}_{ji}$ is the 
expected 
number of $j$-events on a cluster rooted at an $i$-event counting only 
subtrees of 
the form $i - j - j - \ldots - j$ for any number of $j$-events. 
This is thus the expected number of direct $j$-events from an 
$i$-event when also counting subsequent `self-events' $j - j$. It is clear from 
Equation (\ref{eq:Chawkes}) that the results in Section \ref{sec:ip} also hold 
for linear Hawkes processes. If $\rho(\Phi) < 
1$ and the entries of $\Phi$ are nonnegative, then this will also be 
the case for $\bar{\Phi}$ in Equation (\ref{eq:normEq}) 
\citep{mogensen2022equality}. This means that 
the normalized parameters are also within the Hawkes parameter space.

\subsection{Estimation}
\label{ssec:hawkesEst}

		\cite{achab2017} describe how to estimate 
		cumulants of linear Hawkes process. We sketch their approach below. 
		We assume that we observe a linear Hawkes process over the interval 
		$[0, T]$ and that there exists $H > 0$ such that restricting the 
		integration in Equation (\ref{eq:hawkesC}) to $[-H, H]$ introduces only 
		a
		negligible error. As pointed out by \cite{achab2017}, this is 
		reasonable if the support of $\phi_{ji}$ is small compared to $H$ and 
		the spectral radius of $\Phi$ is sufficiently small. Given a 
		realization of a 
		stationary linear Hawkes process on $[0, T]$ let $p_i = \{t_1^i,\ldots, 
		t_{m_i}^i\} \subset [0,T]$ be the observed event times of process $i$. 
		The following are estimators of the first- and second-order cumulants,
		
		\begin{align*}
			\hat{\Lambda}_i = \frac{1}{T} \sum_{k = 1}^{m_i} 1, \ \ \ \ \ \ \ 
			\ \ \ \ \ \ \
			\hat{C}_{ij} = \frac{1}{T} \sum_{k = 1}^{m_i} \left(N_{t_{k}^i + 
				H}^j - N_{t_{k}^i - H}^j - 2H\hat{\Lambda}^j \right)
		\end{align*}
		
		In the above, $N_t^i$ refers to the observed counting process 
		corresponding to process $i$, that is, $N_t^i = 0$ for $t < 0$ and in 
		general $N_t^i = \sum_{k = 1}^{m_i} \mathds{1}_{t_k^i \leq t}$. As 
		noted by \cite{achab2017}, there is a bias in the estimation of the 
		integrated covariance, however, it is found to be negligible. 
		\cite{achab2017} (Theorem 2.1 and Remark 1) show asymptotic consistency 
		for $H_T \rightarrow \infty$ and $H_T^2/T \rightarrow 0$ 
		where $H_T$ is the value of the parameter $H$ used when observing the 
		process on the interval $[0,T]$.

		%%%%%%%%%%%%%%%%%%%%%%%%%%%%%%%%%%%%%%

		\section{Conclusion}
		\label{sec:conclusion}

		 The instrumental variable method in this paper provides a moment 
		 equation for time series models which avoids using a conditional 
		 moment 
		 equation as in 
		 \cite{thams2022}. It also makes minimal assumptions on the marginal 
		 distribution of the instrumental process. On the other hand, it 
		 involves an integral or an
		 infinite sum which needs to be estimated when applying the method. One 
		 should also note that our estimands are slightly different than those 
		 of \cite{thams2022}. As shown, the estimands in this paper do have a 
		 simple causal interpretation, however.
		
		The integrated covariance approach also allows a unified treatment of 
		IV methods in time series (discrete-time) and continuous-time processes 
		as illustrated by the application to the continuous-time linear Hawkes 
		processes. It is clearly of 
		interest to extend this framework to more general classes of 
		continuous-time processes. Finally, one should also note that the 
		parametrization of the integrated 
		covariance can be used to 
		obtain other identification results than the instrumental variable 
		results in this paper.
		
		%%%%%%%%%%%%%%%%%%%%%%%%%%%%%%%%%%%%%%
		
		\acks{This work was supported by a DFF-International Postdoctoral Grant 
		(0164-00023B) from Independent Research Fund Denmark. The author is 
		a member of the ELLIIT Strategic Research Area at Lund University. The 
		author 
		thanks 
		Nikolaj Thams and Jonas Peters for helpful discussions. The author is 
		also grateful to the reviewers for their constructive comments and 
		suggestions.}
		
		%%%%%%%%%%%%%%%%%%%%%%%%%%%%%%%%%%%%%%
		
		\bibliography{C:/Users/swmo/Desktop/-/forsk/references}
		
		%%%%%%%%%%%%%%%%%%%%%%%%%%%%%%%%%%%%%%

			\appendix
			
		\section{Integrated Covariance, VAR(p)}
		\label{app:intCovVARp}
		
First, we rewrite the VAR(p)-process as a 
VAR(1)-process, $Y$, with $n p$ 
coordinate processes,

\begin{align*}
	Y_t &= \Phi_Y Y_{t-1} + \varepsilon_t^Y	=			 \begin{bmatrix}
		\Phi_{1}       & \Phi_{2} & \ldots & \Phi_{p-1} & \Phi_{p} \\
		I       & 0 & \ldots & 0 & 0
		\\
		0 & I & \ldots & 0 & 0 \\
		\vdots & \vdots & \ddots & \vdots & \vdots \\
		0      & 0 & \ldots & I & 0
		\end{bmatrix} Y_{t-1} + 
		\begin{bmatrix}
			\varepsilon_t     \\  
			0 \\
			0\\
			\vdots \\
			0   
			\end{bmatrix}
			\end{align*}
			
			\noindent The VAR(1)-computations from above still hold which means 
			that the 
			integrated covariance of $Y$ can be written as
			
			$$
			C_Y = (I_{np} - \Phi_Y)^{-1} \Theta_Y (I_{np} - \Phi_Y)^{-T}.
			$$
			
			\noindent Note that if 1 is an eigenvalue of $\Phi_Y$, then it 
			corresponds to an eigenvector which is the concatenation of $p$ 
			copies 
			of a single $n$-vector. This $n$-vector is also an eigenvector of 
			$\Phi = \sum_{i=1}^p \Phi_i$ with eigenvalue 1 which is a 
			contradiction. This means that 
			$I_{np} - 
			\Phi_Y$ is invertible. We can use Schur complements and the 
			structure 
			of $(I_{np} - 
			\Phi_Y)$ to see 
			that $((I_{np} - \Phi_Y)^{-1})_{1:n,1:n} = (I_n - \Phi)^{-1}$ where 
			$\Phi 
			= 
			\sum_{i=1}^p \Phi_i$. From the sparsity of $\Theta_Y$ it follows 
			that 
			
			$$
			C = (C_Y)_{1:n,1:n} = (I_n - \Phi)^{-1}\Theta (I_n - \Phi)^{-T}.
			$$

		\section{Normalization}
		\label{app:norm}
		
		\begin{exmp}	
			Consider a representation such that $\Phi$ is normalized (i.e., has 
			zeros on the diagonal)
			
			$$
			C = (I_n - \Phi)^{-1}\Theta (I_n - \Phi)^{-T}.
			$$
			
			\noindent For any diagonal matrix such that $D_{ii} \neq 1$ for 
			all $i$,
			
			$$
			C = (D(I_n - \Phi))^{-1}D\Theta D (D(I_n - {\Phi}))^{-T} = 
			(I_n 
			- \bar{\Phi})^{-1}\bar{\Theta}  ((I_n - \bar{\Phi}))^{-T}.
			$$
			
			\noindent If $0 < D_{ii} < 1$, $\Phi$ is nonnegative, and 
			$\rho(\Phi) < 1$, we show that 
			$\rho(\bar{\Phi}) < 1$. To see this note 
			that $\bar{\Phi} = I_n - D + D\Phi$. This is a nonnegative matrix 
			and 
			let $\lambda = \rho(\bar{\Phi})$. A nonzero vector 
			$x$ with nonnegative entries can be 
			chosen such that $\bar{\Phi}x = \lambda x$ \citep[Theorem 
			8.3.1]{horn1985}. $\Phi$ and $x$ have nonnegative entries and $x$ 
			is nonzero and therefore $\Phi x \geq x$ (the inequalities are to 
			be 
			read entrywise) implies that 
			$\rho(\Phi) \geq 
			1$ \citep[Theorem 8.3.2]{horn1985} so $(\Phi x)_i < x_i$ for some 
			$i$. 
			We have $\lambda x = (I_n - D + D\Phi)x$ and therefore $(\lambda 
			x)_i 
			< 
			x_i$ so $\lambda < 1$. 
			We see that $D\Theta D$ is positive 
			definite. This shows that we cannot identify unnormalized direct 
			effects from the integrated covariance matrix as every nonzero 
			entry of $\bar{\Phi}$ is different from the corresponding entry of 
			$\Phi$ (note the diagonal of $\bar{\Phi}$ is nonzero),
			
			$$
			\bar{\Phi}_{ii} = 1 - D_{ii} + \sum_k D_{ik}\Phi_{ki} = 1 - D_{ii} 
			> 0
			$$
			
			\noindent and for $i \neq j$,
			
			$$
			\bar{\Phi}_{ij} = \sum_k D_{ik}\Phi_{kj} = D_{ii}\Phi_{ij} 
			<\Phi_{ij}
			$$
			
			\noindent when $\Phi_{ij}\neq 0$.
			
			For the time series case, note also that
			
			\begin{align*}
				\bar{\Phi}_{BB} = I_b - D_{BB} + (D\Phi)_{BB} = I_b - D_{BB} + 
				D_{BB}\Phi_{BB}
			\end{align*}
			
			\noindent and when $x$ is a nonzero vector such that $x = 
			\bar{\Phi}_{BB}x$ 
			then
			
			\begin{align*}
				x = \bar{\Phi}_{BB}x = (I_b - D_{BB} + 
				D_{BB}\Phi_{BB})x
			\end{align*}
			
			\noindent This implies $D_{BB}x = D_{BB}\Phi_{BB}x$ and $x = 
			\Phi_{BB}x$ so 
			$1$ is an eigenvalue of $\Phi_{BB}$ and therefore $I_b - \Phi_{BB}$ 
			is not invertible which is a contradiction. Therefore $1$ is also 
			not an eigenvalue of $\bar{\Phi}_{BB}$ and it follows that $I_b - 
			\bar{\Phi}_{BB}$ is invertible.
			
			\label{exmp:norm}
			\end{exmp}
			
	\section{Proofs}
	\label{app:proofs}
	
			\begin{proof}[Theorem \ref{thm:uniIV}]
				We have $\rho(\Phi) < 1$, and we define $R = (I_n - \Phi)^{-1} 
				= \sum_{k = 0}^{\infty} \Phi^k$. As $\iota$ is exogenous, we 
				have that $C_{\alpha\iota} = 
				R_{\iota\iota}\Theta_{\iota\iota}R_{\alpha\iota}$ and 
				$C_{\beta\iota} = 
				R_{\iota\iota}\Theta_{\iota\iota}R_{\beta\iota}$. From the 
				definition of an instrumental process, we have 
				$C_{\alpha\iota}\neq 
				0$ and $R_{\alpha\iota}\neq 0$, and therefore 
				$R_{\beta\iota}/R_{\alpha\iota}$ is identified. Using that 
				$\iota$ is an instrumental process and 
				the fact that $I_n 
				= 
				(I_n - \Phi)R$, it follows that $R_{\beta\iota} = 
				\Phi_{\beta\alpha}R_{\alpha\iota} + 
				\Phi_{\beta\beta}R_{\beta\iota}$. Therefore 
				$R_{\beta\iota}/R_{\alpha\iota} = 
				\Phi_{\beta\alpha}/(1 - \Phi_{\beta\beta})$.
				\end{proof}
		
				\begin{proof}[Theorem \ref{thm:multiIV}]
					From exogeneity of $I$, it holds that $C_{BI} = 
					R_{BI}\Theta_{II}R_{II}$ and $C_{AI} = 
					R_{AI}\Theta_{II}R_{II}$. The matrix $C_{AI}$ has full rank 
					and it is therefore invertible since $\vert A\vert = \vert 
					I\vert$. 
					Matrices $R_{AI}$, $\Theta_{II}$, and $R_{II}$ are 
					therefore also invertible. In that case,
					
					$$
					R_{BI}(R_{AI})^{-1} = C_{BI}(C_{AI})^{-1}
					$$
					
					\noindent and therefore $R_{BI}(R_{AI})^{-1}$ is 
					identified. From 
					the 
					definition of $R$, we see that $I_n = (I_n - \Phi)R$ and 
					therefore 
					$R = 
					I_n 
					+ \Phi R$. This means that 
					
					$$
					R_{BI} = \sum_C \Phi_{BC}R_{CI} = \Phi_{BA}R_{AI} + 
					\Phi_{BB}R_{BI}.
					$$
					
					\noindent The last equality uses that $I$ is an 
					instrumental process. 
					We 
					obtain 
					
					$$
					R_{BI}(R_{AI})^{-1} = (I_b - \Phi_{BB})^{-1}\Phi_{BA}.
					$$
					
					\noindent Note that $R_{AI}$ is invertible as noted as 
					above. In 
					the linear Hawkes case, it 
					holds that $\rho(\Phi_{BB}) \leq \rho(\Phi) < 1$ 
					\citep[Corollary 8.1.20]{horn1985} so 
					$I_b - \Phi_{BB}$ is also invertible.
				\end{proof}

	\begin{proof}[Proof of Theorem \ref{thm:genTimeSeries}, general $p$]
				
				We see that
				
				\begin{align*}
					& E(X_{t}^B (X_{t+k}^I)^T) = \\ & E\left(\left(\sum_j 
					\Phi_{j,BA}X_{t-j}^A 
					+ \sum_j \Phi_{j,BB}X_{t-j}^B + 
					g_B(\ldots, 	X_{t-2}^U, X_{t-1}^U, 
					\varepsilon_{t}^U)\right) (X_{t+k}^I)^T 
					\right).
				\end{align*}
				
				\noindent We sum over $k$,
				
				\begin{align*}
					\sum_{k=-\infty}^{\infty}  E(X_{t}^B 
					(X_{t+k}^I)^T) &= 
					\sum_j 
					\Phi_{j,BA} \sum_{k=-\infty}^{\infty} E(X_{t-j}^A 
					(X_{t+k}^I)^T ) 
					\\
					&+ \sum_j \Phi_{j,BB} \sum_{k=-\infty}^{\infty} 
					E(X_{t-j}^B 
					(X_{t+k}^I)^T ) \\
					&= 		\sum_j 
					\Phi_{j,BA} \sum_{k=-\infty}^{\infty} E(X_{t}^A 
					(X_{t+k}^I)^T 
					) 
					\\
					&+ \sum_j \Phi_{j,BB} \sum_{k=-\infty}^{\infty} 
					E(X_{t}^B 
					(X_{t+k}^I)^T ).
				\end{align*}
				
				\noindent From this it follows that $C_{BI} = (I_b - 
				\Phi_{BB})^{-1}\Phi_{BA}C_{AI}$.
	\end{proof}
		
	\section{Numerical Examples}
	\label{app:numerical}
	
	\begin{figure}
		\includegraphics[scale = .5]{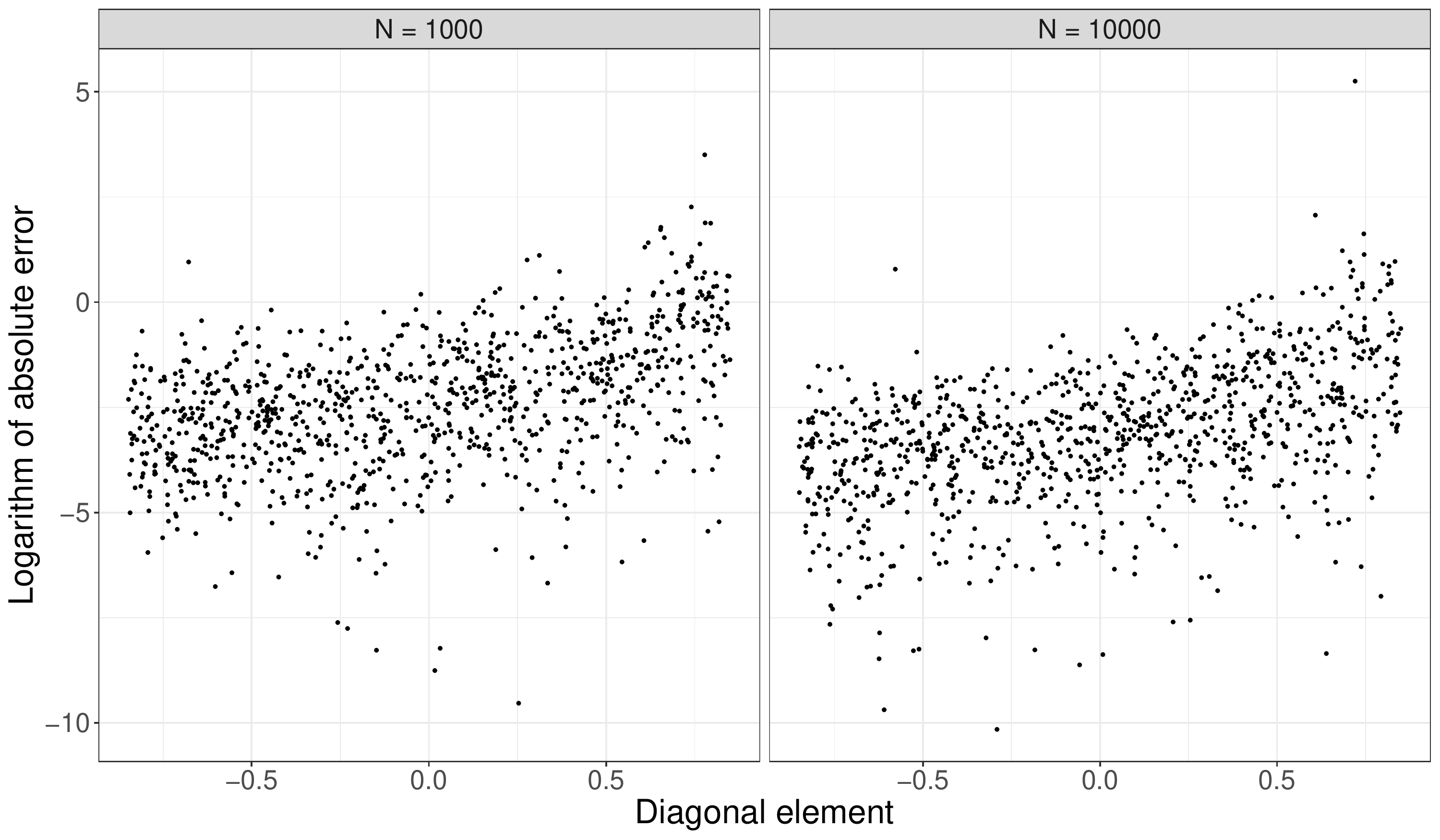}
		\caption{Results from Experiment E7. The vertical axis is the logarithm 
		of $\vert \hat{\theta}_i^7 - {\theta}_i^7 \vert$. The horizontal axis 
		is $\Phi_{BB}$, $B = \{3\}$. As expected from the identifying formula, 
		estimation 
		accuracy 
		deteriorates 
		when this value increases as also reflected in E7 in 
		Table \ref{tab:res}. Note that Experiments E1 and E7 are identical 
		except for the fact that the 
		diagonal elements are sampled from a larger interval in E7. This leads 
		to smaller denominators in the definition of the normalized 
		parameter(s) and worse estimation using a naive estimator.}
		\label{fig:sim}
	\end{figure}
	
	We list results from numerical experiments in this section. In Experiments 
	E1-E7 we generated observations from a 
	single 
	time series of length $N$ and estimated the normalized parameter(s) using 
	the plug-in estimator (see Section \ref{sec:est}). We repeated this $m$ 
	times.
	Table 
	\ref{tab:res} reports empirical mean squared error (MSE), 
	$(1/m)\cdot\sum_{i = 
	1}^m 
	\lVert \hat{\theta}_i^j - {\theta}_i^j \rVert^2 $, where $\lVert \cdot 
	\rVert$ is the Euclidean norm, ${\theta}_i^j$ is the true normalized 
	parameter(s) in the $i$'th run of the $j$'th experiment, and 
	$\hat{\theta}_i^j$ is its estimate. Experiment H1 is a linear Hawkes 
	example. The causal graphs do not include loops, $\alpha\rightarrow\alpha$, 
	but note 
	that all diagonal parameters, e.g., the diagonal of $\Phi$ in a 
	VAR(1)-process, were nonzero. In the time series experiments, we used the 
	{\tt sandwich} package in {\tt R} to estimate the 
	long-run covariance \citep{sandwichRpackage, R}.
	
	\begin{table}
		\centering
		\begin{tabular}{ c c c c c }
			& N & m & MSE & R \\ \hline\hline
			E1 & 1000 & 100 & $0.027$ & $[-1.57,1.45]$ \\ 
			& 10000 & 100 & $0.012$ & $[-1.28,1.63]$ \\  \hline
			E2 & 1000 & 100 & $0.072$ & $[-1.38,1.58]$ \\ 
			& 10000 & 100 & $0.005$ & $[-1.38,1.23]$ \\  \hline
			E3 & 1000 & 100 & $0.955$ & $[-1.70,1.72]$ \\ 
			& 10000 & 100 & $0.061$ & $[-1.68,1.80]$ \\  \hline
			E4 & 1000 & 100 & $0.055$ & $[-1.39,1.44]$ \\ 
			& 10000 & 100 & $0.004$ & $[-1.35,1.86]$ \\  \hline
			E5 & 1000 & 100 & $0.434$ & $[-5.04,6.99]$ \\ 
			& 10000 & 100 & $0.214$ & $[-11.16,8.94]$ \\  \hline
			E6 & 1000 & 100 & $0.203$ & $[1.25,1.25]$ \\ 
			& 10000 & 100 & $0.020$ & $[1.25,1.25]$ \\  \hline
			E7 & 1000 & 1000 & $1.606$ & $[-5.280,6.251]$ \\ 
			& 10000 & 1000 & $36.803$ & $[-6.507,5.276]$ \\  \hline
			H1 &  & 500 & $0.039$ & $[0.26, 0.96]$ \\  \hline 
		\end{tabular}
		\caption{Results from the experiments described in Section 
			\ref{app:numerical}. $N$ is the length of the observed time series 
			from 
			which the estimate is computed and $m$ is the number of 
			repetitions. R is 
			the range of the true normalized parameter(s) over the $m$ 
			runs of each experiment. The true normalized parameter was fixed in 
			Experiment E6. We see reasonably good performance, except for E7 
			where a 
			larger sampling interval for diagonal elements creates very large 
			errors in 
			some instances (see Figure \ref{fig:sim}).}
		\label{tab:res}
	\end{table}
	
	\begin{figure}
		%\begin{tabular}{cc}
		\hspace{.04\linewidth}\begin{minipage}{0.4\linewidth}
			\centering
			\begin{tikzpicture}[scale=0.7]
			\tikzset{vertex/.style = 
				{shape=circle,draw,minimum 
					size=1.5em, 
					inner 
					sep = 0pt}}
			\tikzset{edge/.style = {->,> = latex', 
					thick}}
			\tikzset{edgebi/.style = {<->,> = 
					latex', 
					thick}}
			\tikzset{every loop/.style={min 
					distance=8mm, 
					looseness=5}}
			\tikzset{vertexFac/.style = 
				{shape=rectangle,draw,minimum 
					size=1.5em, 
					inner sep = 0pt}}
			
			% vertices
			%\draw [line width=35pt,opacity=0.1, 
			%blue,line 
			%cap=round,rounded
			%corners] (0,0.5) -- (0,2) -- 
			%(-1.5,1.5) -- 
			%(0,0.5);
			\node[vertex] (a) at  (-4,0) {$1$};
			\node[vertex] (b) at  (-4,-2) {$2$};
			\node[vertex] (c) at  (-1,-2) {$3$};
			\node[vertex] (d) at  (-1,0) {$4$};
			\node[vertex] (e) at  (2,0) {$5$};
			\node[vertexFac] (f) at  (2,1.75) {$6$};
			
			\node at (-3.5,1.75) {\Large 
				\textbf{A}};
			
			%edges
			
			%\draw[edge] (a) to [bend left = 15] 
			%(b);
			%\draw[edge] (b) to [bend left = 15] 
			%(a);
			\draw[edge] (a) to (c);
			\draw[edge] (b) to (c);
			\draw[edge] (a) to (d);
			\draw[edge] (b) to (d);
			%\draw[edge] (d) to [bend left = 15] 
			%(c);
			%\draw[edge] (c) to [bend left = 15] 
			%(d);
			\draw[edge] (c) to (e);
			\draw[edge] (d) to (e);
			\draw[edge] (f) to (e);
			\draw[edge] (f) to (c);
			\draw[edge] (f) to (d);
			%\draw[edge] (e) to [bend left] (c);
			
			\end{tikzpicture}
		\end{minipage}\hfill%
		\begin{minipage}{0.4\linewidth}
			\centering
			\begin{tikzpicture}[scale=0.7]
			\tikzset{vertex/.style = 
				{shape=circle,draw,minimum 
					size=1.5em, 
					inner 
					sep = 0pt}}
			\tikzset{edge/.style = {->,> = latex', 
					thick}}
			\tikzset{edgebi/.style = {<->,> = 
					latex', 
					thick}}
			\tikzset{every loop/.style={min 
					distance=8mm, 
					looseness=5}}
			\tikzset{vertexFac/.style = 
				{shape=rectangle,draw,minimum 
					size=1.5em, 
					inner sep = 0pt}}
			
			% vertices
			%\draw [line width=35pt,opacity=0.1, 
			%blue,line 
			%cap=round,rounded
			%corners] (0,0.5) -- (0,2) -- 
			%(-1.5,1.5) -- 
			%(0,0.5);
			\node[vertex] (a) at  (-4,0) {$1$};
			\node[vertex] (b) at  (-4,-2) {$2$};
			\node[vertex] (c) at  (-1,0) {$3$};
			\node[vertex] (d) at  (2,0) {$4$};
			\node[vertexFac] (e) at  (2,1.75) {$5$};
			
			\node at (-3.5,1.75) {\Large 
				\textbf{B}};
			
			%edges
			
			\draw[edge] (a) to [bend left = 15] 
			(b);
			\draw[edge] (b) to [bend left = 15] 
			(a);
			\draw[edge] (a) to (c);
			\draw[edge] (b) to (c);
			\draw[edge] (d) to [bend left = 25] 
			(c);
			\draw[edge] (c) to [bend left = 0] 
			(d);
			\draw[edge] (e) to (c);
			\draw[edge] (e) to (d);
			%\draw[edge] (e) to [bend left] (c);
			\end{tikzpicture}
		\end{minipage}\hspace{.04\linewidth}
		%\end{tabular}
		\caption{Graphs from Experiments E3 (graph 
			\textbf{A}) and E4 (graph \textbf{B}) in 
			Section \ref{app:numerical}. Square nodes 
			correspond to unobserved processes.}
		\label{fig:experiments}
	\end{figure}
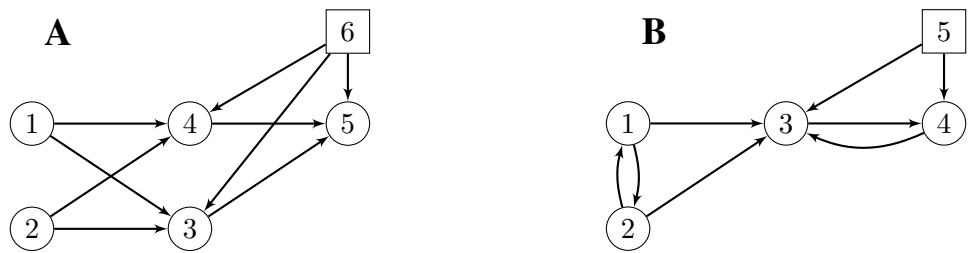

	\paragraph{Experiment E1}
	Four-dimensional VAR(1)-process corresponding to graph \textbf{A} in Figure 
	\ref{fig:IVexample}. Nonzero, nondiagonal VAR(1)-parameters were sampled 
	uniformly in $[-1,-0.2] \cup [0.2,1]$. Diagonal para\-meters were nonzero 
	and 
	sampled 
	uniformly in $[-0.5,.5]$. The parameters were sampled repeatedly until 
	parameters satisfying the stability condition of Subsection \ref{ssec:ts} 
	were obtained.
	
	\paragraph{Experiment E2}
	Four-dimensional VAR(1)-process corresponding to graph \textbf{B} in Figure 
	\ref{fig:IVgraphs}. Nonzero parameters were sampled as in Experiment E1.
	
	\paragraph{Experiment E3}
	Six-dimensional VAR(1)-process corresponding to graph \textbf{A} in Figure 
	\ref{fig:experiments}. Nonzero parameters were sampled as in Experiment E1. 
	The set $\{1,2\}$ is instrumental for the effect from $\{3,4\}$ to $5$.
	
	\paragraph{Experiment E4}
	Five-dimensional VAR(1)-process corresponding to graph \textbf{B} in Figure 
	\ref{fig:experiments}. Nonzero parameters were sampled as in Experiment E1. 
	The set $\{1,2\}$ is instrumental for the effect from $3$ to $4$. We used 
	$W = I_2$.
	
	\paragraph{Experiment E5}
	Four-dimensional VAR(2)-process corresponding to graph \textbf{A} in 
	Figure 
	\ref{fig:IVexample}. Nonzero parameters were sampled as in Experiment E1 
	and 
	sampling was repeated until the parameters satis\-fied $(\Phi_1 + 
	\Phi_2)_{21} \geq 0.2$ and the signs of $(\Phi_1)_{21}$ and 
	$(\Phi_2)_{21}$ were equal.
	
	\paragraph{Experiment E6}
	Four-dimensional time series corresponding to the framework in Section 
	\ref{sec:confoundingts}. Variables 
	$\varepsilon_t^I, \varepsilon_t^A, \varepsilon_t^B$, and $\varepsilon_t^U$ 
	were sampled as independent Gaussian variables with a standard deviation of 
	$0.25$. For each $t$, data was generated as
	
	\begin{align*}
		X_t^I &= -\frac{1}{1 + (X_{t-1}^I)^2} + \varepsilon_t^I, \\
		X_t^U &= \frac{\exp(X_{t-1}^U)}{1+\exp(X_{t-1}^U)} + \varepsilon_t^U, \\
		X_t^A &= -\frac{3}{1 + \exp(X_{t-1}^I)} - 0.5\cdot X_{t-1}^A + 
		X_{t-1}^U\cdot \varepsilon_t^A, \\
		X_t^B &= 0.5\cdot X_{t-1}^A + 0.6\cdot X_{t-1}^B + X_{t-1}^U\cdot 
		\varepsilon_t^B.
	\end{align*}
	
	\paragraph{Experiment E7}
	Four-dimensional VAR(1)-process corresponding to graph \textbf{A} in Figure 
\ref{fig:IVexample}. Nonzero, nondiagonal parameters were sampled as in 
	Experiment E1. Diagonal parameters were nonzero and 
	sampled 
	uniformly in $[-0.85,.85]$.
	
	\paragraph{Experiment H1}
	
	We generated observations from a linear Hawkes process corresponding to 
	graph \textbf{A} in Figure \ref{fig:IVexample} using the {\tt hawkes} 
	package in {\tt R} \citep{hawkesRpackage, R}. In their parametrization, 
	$\beta$-parameters were equal to $1$, and we sampled $\alpha$-parameters 
	uniformly 
	on $[0.2,0.5]$. In a single run, each coordinate process had between 10000 
	and 40000 events.

	\section{A VAR(2)-example}
	\label{app:var2exmp}
	
	We give an example of a VAR(2)-process satisfying the assumptions of the 
	instrumental process method (see also Figure \ref{fig:var2exmp}). We assume 
	that $X_t = (X_t^1, X_t^2, X_t^3, 
	X_t^4)^T$ is a VAR(2)-process such that
	
	\begin{align}
		X_t &= \Phi_1 X_{t-1} + \Phi_2 X_{t-2} + \varepsilon_t \\ 
		&= \begin{bmatrix}
		\Phi_{11}^1       & 0 & 0 & 0 \\
		\Phi_{21}^1       & \Phi_{22}^1 & \Phi_{23}^1 & \Phi_{24}^1 
		\\
		0 & \Phi_{32}^1 & \Phi_{33}^1 & \Phi_{34}^1 \\
		0      & 0 & 0 & \Phi_{44}^1
		\end{bmatrix} X_{t-1} + \begin{bmatrix}
		\Phi_{11}^2       & 0 & 0 & 0 \\
		\Phi_{21}^2       & 0 & 0 & 0 
		\\
		0 & 0 & 0 & \Phi_{34}^2 \\
		0      & 0 & 0 & \Phi_{44}^2
		\end{bmatrix} X_{t-2} + \varepsilon_t.
		\label{eq:var2exmp}
	\end{align}
	
	\noindent We use $\Phi_{ij}^k $ to denote the $(i,j)$-entry of $\Phi_k$. 
	Matrices $\Phi_1$ and $\Phi_2$ are as defined in (\ref{eq:var2exmp}) and 
	$\Phi_k = 0$ for all $k\neq 1,2$. Graph \textbf{B} in Figure 
	\ref{fig:var2exmp} has an edge $X_{k_1}^i 
	\rightarrow X_{k_2}^j$ if $\Phi_{ji}^{k_2-k_1}$ may be nonzero. Graph 
	\textbf{A} is a rolled version of graph \textbf{B} and we have that $i 
	\rightarrow j$ in graph \textbf{A}, $i\neq j$, if and only if there exists 
	$k_1$ and 
	$k_2$ such that $X_{k_1}^i \rightarrow X_{k_2}^j$ in graph \textbf{B}. Let 
	$\Phi = \Phi_1 + 
	\Phi_2$. We 
	see that process $1$ is an instrumental process for $2\rightarrow 3$ 
	(Definition 
	\ref{def:uniIV}), and this means that $\Phi_{32}/(1 - \Phi_{33})$ is 
	identified from the observed integrated covariance if $C_{21} \neq 0$.

								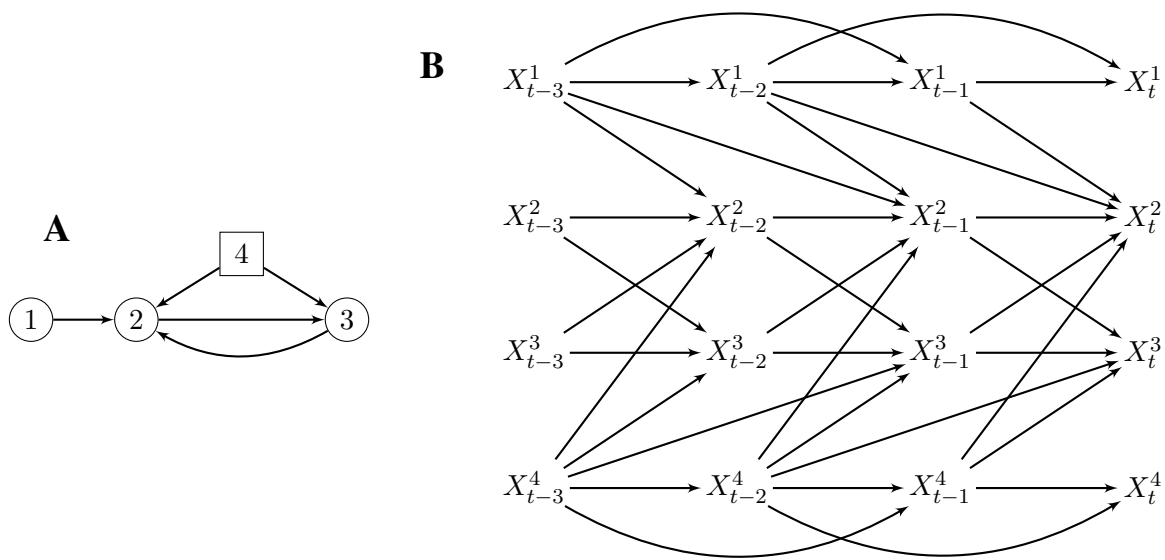
\begin{figure}
									%\begin{tabular}{cc}
									\begin{minipage}[.7\textheight]{0.3\linewidth}
										\centering
										\begin{tikzpicture}[scale=0.7]
										\tikzset{vertex/.style
											= 
											{shape=circle,draw,minimum 
												size=1.5em, 
												inner sep = 0pt}}
										\tikzset{edge/.style= {->,> = 
												latex',thick}}
										\tikzset{edgebi/.style= {<->,> = 
												latex', 
												thick}}
										\tikzset{every 
											loop/.style={min distance=8mm, 
												looseness=5}}
										\tikzset{vertexFac/.style= 	
											{shape=rectangle,draw,minimum 
												size=1.5em, 
												inner sep = 
												0pt}}

										% graph A
										\def\x{-6}
										\node[vertex] (i2) at (-4,0+\x) 	
										{$1$};
										\node[vertex] (a2) at  	(-2,0+\x) {$2$};
										\node[vertex] (b2) 	at  (2,0+\x) 	
										{$3$};
										\node[vertexFac] 	(u2) 	at  	
										(0,1.25+\x)	
										{$4$};
										
										\node at 
										(-3.5,1.75+\x) {\Large 
											\textbf{A}};
										
										%edges
										
										\draw[edge] (i2) to 
										(a2);
										\draw[edge] (a2) to 
										(b2);				
										\draw[edge] (u2) to 
										(a2);
										\draw[edge] (u2) to 
										(b2);
										\draw[edge, bend left] (b2) to 
										(a2);

										\end{tikzpicture}
									\end{minipage}\hspace{.05\linewidth}%
									\begin{minipage}{0.65\linewidth}
										\centering
										\begin{tikzpicture}[scale=0.9]
										\tikzset{vertex/.style = 
											{shape=circle,minimum 
												size=1.5em, 
												inner 
												sep = 0pt}}
										\tikzset{edge/.style = {->,> = latex', 
												thick}}
										\tikzset{edgebi/.style = {<->,> = 
												latex', 
												thick}}
										\tikzset{every loop/.style={min 
												distance=8mm, 
												looseness=5}}
										\tikzset{vertexFac/.style = 
											{shape=rectangle,draw,minimum 
												size=1.5em, 
												inner sep = 0pt}}
										
										% vertices
										%\draw [line width=35pt,opacity=0.1, 
										%blue,line 
										%cap=round,rounded
										%corners] (0,0.5) -- (0,2) -- 
										%(-1.5,1.5) -- 
										%(0,0.5);
										\def\x{3}
										
										\node[vertex] (i1) at  (-2*\x,0) 
										{$X_{t-3}^1$};
										\node[vertex] (i2) at  (-\x,0) 
										{$X_{t-2}^1$};
										\node[vertex] (i3) at  (0,0) 
										{$X_{t-1}^1$};
										\node[vertex] (i4) at  (\x,0) 
										{$X_{t}^1$};
										\node[vertex] (a1) at  (-2*\x,-2) 
										{$X_{t-3}^2$};
										\node[vertex] (a2) at  (-\x,-2) 
										{$X_{t-2}^2$};
										\node[vertex] (a3) at  (0,-2) 
										{$X_{t-1}^2$};
										\node[vertex] (a4) at  (\x,-2) 
										{$X_{t}^2$};
										\node[vertex] (b1) at  (-2*\x,-4) 
										{$X_{t-3}^3$};
										\node[vertex] (b2) at  (-\x,-4) 
										{$X_{t-2}^3$};
										\node[vertex] (b3) at  (0,-4) 
										{$X_{t-1}^3$};
										\node[vertex] (b4) at  (\x,-4) 
										{$X_{t}^3$};
										\node[vertex] (u1) at  (-2*\x,-6) 
										{$X_{t-3}^4$};
										\node[vertex] (u2) at  (-\x,-6) 
										{$X_{t-2}^4$};
										\node[vertex] (u3) at  (0,-6) 
										{$X_{t-1}^4$};
										\node[vertex] (u4) at  (\x,-6) 
										{$X_{t}^4$};

										\node at (-7.5,0.25) {\Large 
											\textbf{B}};
										
										%edges
										\draw[edge] (i1) to 
										(i2);
										\draw[edge] (i2) to  
										(i3);
										\draw[edge] (i3) to 
										(i4);
										\draw[edge] (a1) to 
										(a2);
										\draw[edge] (a2) to  
										(a3);
										\draw[edge] (a3) to 
										(a4);
										\draw[edge] (b1) to 
										(b2);
										\draw[edge] (b2) to  
										(b3);
										\draw[edge] (b3) to 
										(b4);
										\draw[edge] (u1) to 
										(u2);
										\draw[edge] (u2) to  
										(u3);
										\draw[edge] (u3) to 
										(u4);
										
										\draw[edge] (i1) to 
										(a2);
										\draw[edge] (i2) to  
										(a3);
										\draw[edge] (i3) to 
										(a4);
										\draw[edge] (a1) to 
										(b2);
										\draw[edge] (a2) to  
										(b3);
										\draw[edge] (a3) to 
										(b4);
										\draw[edge] (u1) to 
										(a2);
										\draw[edge] (u2) to  
										(a3);
										\draw[edge] (u3) to 
										(a4);
										\draw[edge] (u1) to 
										(b2);
										\draw[edge] (u2) to  
										(b3);
										\draw[edge] (u3) to 
										(b4);
										\draw[edge] (b1) to 
										(a2);
										\draw[edge] (b2) to  
										(a3);
										\draw[edge] (b3) to 
										(a4);
									% edges across two lags
									\draw[edge] (i1) [bend left] to 
									(i3);
									\draw[edge] (i2) [bend left] to 
									(i4);
									\draw[edge] (i1) to 
									(a3);
									\draw[edge] (i2) to 
									(a4);
									\draw[edge] (u1) [bend right] to 
									(u3);
									\draw[edge] (u2) [bend right] to 
									(u4);
									\draw[edge] (u1) to 
									(b3);
									\draw[edge] (u2) to 
									(b4);
										
										\end{tikzpicture}
									\end{minipage}
									
									%\end{tabular}
									\caption{Graphs from the example in Section 
									\ref{app:var2exmp}. Process $4$ is 
									unobserved.}
									\label{fig:var2exmp}
								\end{figure}

			\listoffixmes
			
	\end{document}